\documentclass[reqno]{amsart}
\usepackage{amsmath}
\usepackage{amsthm, amscd, amssymb, amsfonts, amsbsy}
\usepackage[usenames, dvipsnames]{color}
\usepackage{enumerate}
\usepackage{esint}
\usepackage{hyperref}
\usepackage{mathrsfs}
\usepackage{todonotes}
\usepackage{verbatim}

\numberwithin{equation}{section}

\theoremstyle{plain}
\newtheorem{theorem}{Theorem}[section]
\newtheorem{lemma}[theorem]{Lemma}
\newtheorem{corollary}[theorem]{Corollary}

\theoremstyle{definition}
\newtheorem{definition}[theorem]{Definition}

\theoremstyle{remark}
\newtheorem{remark}[theorem]{Remark}
\newtheorem{notation}{Notation}[section]

\def\bR{\mathbb{R}}

\def\cA{\mathcal{A}}
\def\cD{\mathcal{D}}
\def\sD{\mathsf{D}}

\def\cL{\mathcal{L}}
\def\cF{\mathcal{F}}
\def\cE{\mathcal{E}}
\def\cX{\mathcal{X}}
\def\cY{\mathcal{Y}}
\def\cZ{\mathcal{Z}}

\newcommand{\Div}{\operatorname{div}}

\begin{document}
\title[Gradient estimates]{Gradient estimates for Stokes and Navier-Stokes systems with piecewise DMO coefficients}

\author[J. Choi]{Jongkeun Choi}
\address[J. Choi]{Department of Mathematics Education, Pusan National University, Busan 46241, Republic of Korea}
\email{jongkeun\_choi@pusan.ac.kr}
\thanks{
J. Choi was supported by  National Research Foundation of Korea(NRF) under agreement NRF-2019R1F1A1058826}

\author[H. Dong]{Hongjie Dong}
\address[H. Dong]{Division of Applied Mathematics, Brown University, 182 George Street, Providence, RI 02912, USA}
\email{Hongjie\_Dong@brown.edu }
\thanks{H. Dong was partially supported by the Simons Foundation, grant no. 709545}

\author[L. Xu]{Longjuan Xu}
\address[L. Xu]{Department of Mathematics, National University of Singapore, 10 Lower Kent Ridge Road, Singapore 119076}
\email{ljxu311@163.com}
\thanks{}


\begin{abstract}
We study stationary Stokes systems in divergence form with piecewise Dini mean oscillation coefficients and data in a bounded domain containing a finite number of subdomains with $C^{1,\rm{Dini}}$ boundaries.
We prove that if $(u, p)$ is a weak solution of the system, then $(Du, p)$ is bounded and piecewise continuous.
The corresponding results for stationary Navier-Stokes systems are also established, from which the Lipschitz regularity of the stationary $H^1$-weak solution in dimensions $d=2,3,4$ is obtained.
\end{abstract}

\maketitle

\section{Introduction}

In this paper, we consider stationary Stokes systems with variable coefficients
\begin{align}\label{stokes}
\begin{cases}
\mathcal{L}u+\nabla p=D_{\alpha}f_{\alpha}&\quad\mbox{in }~\cD,\\
\Div u=g&\quad\mbox{in }~\cD.
\end{cases}
\end{align}
The differential operator $\cL$ is in divergence form acting on column vector valued functions $u=(u^1,\ldots, u^d)^{\top}$ as follows:
\begin{equation}\label{210322@eq1}
\cL u=D_\alpha (A^{\alpha\beta}D_\beta u),
\end{equation}
where we use  the Einstein summation convention over repeated indices.
The domain $\cD$ is bounded in $\bR^d$ which consists of a finite number of disjoint subdomains and  the coefficients $A^{\alpha\beta}=A^{\alpha\beta}(x)$ can have jump discontinuities along the boundaries of the subdomains.
As is well known, such a system is partly motivated by the study of composite materials with closely spaced interfacial boundaries.
We refer the reader to \cite{MN1987,MR3078336} for Stokes flow over composite spheres. Moreover, it can be used to model  the motion of inhomogeneous fluids with density dependent viscosity and  multiple fluids with interfacial boundaries; see  \cite{MR0425391, MR1422251, MR2663713, MR3758532} and the references therein. Another direction is the study of stress concentration in high-contrast composites with densely packed inclusions whose material properties differ from that of the background.
In \cite{AKKY20}, Ammari et al. investigated the stress concentration of  Stokes flow between adjacent circular cylinders.

 There is a large body of literature concerning regularity theory for partial differential equations arising from the problems of composite materials.
For the theory of elliptic equations/systems in divergence form, we refer the reader to Chipot et al. \cite{MR0853976}, Li-Vogelius \cite{MR1770682}, Li-Nirenberg \cite{MR1990481}, Dong-Li \cite{MR3902466},  and the references therein.
In particular, $W^{1, \infty}$ and piecewise $C^{1,\delta'}$-estimates were obtained by Li-Nirenberg \cite{MR1990481} for elliptic systems with piecewise $C^{\delta}$ coefficients in a domain which consists of a finite number of disjoint subdomains with $C^{1,\mu}$ boundaries, $0<\mu\leq1$ and $0<\delta'\leq\min\{\delta,\frac{\mu}{2(1+\mu)}\}$.
The results in \cite{MR1990481} were extended by the second and third named authors \cite{MR3961984} to the system with piecewise Dini mean oscillation (DMO) coefficients and subdomains having $C^{1, \rm{Dini}}$ boundaries.
They also established piecewise $C^{1,\delta'}$-estimate for solutions under the same conditions and with $0<\delta'\leq\min\{\delta,\frac{\mu}{1+\mu}\}$.
See also \cite{dx2021} for the corresponding results for parabolic systems.
For further related results, one can refer to \cite{MR2927619,MR3403998} for parabolic and elliptic systems with partially Dini or  H\"older continuous coefficients and \cite{MR3912724} for Stokes systems with partially Dini mean oscillation coefficients.

Inspired by the work mentioned above, we are interested in gradient estimates for Stokes systems with piecewise DMO coefficients. The goal of this paper consists of two aspects. We first extend the results in \cite{MR3961984} for elliptic systems to  the stationary Stokes systems \eqref{stokes}. Precisely, we show in Theorem \ref{M1} that
if the coefficients and data are of piecewise Dini mean oscillation and the boundaries of subdomains are $C^{1, \rm{Dini}}$, then
for every weak solution $(u,p)\in W^{1,q}(\cD)^d\times L^q(\cD)$ to \eqref{stokes}, $q\in(1,\infty)$, $Du$ and $p$ are locally bounded and piecewise continuous.
As an application, we obtain piecewise H\"older continuity for $Du$ and $p$ under H\"older regularity assumptions on the coefficients and the boundaries of the subdomains.
We remark that the corresponding estimates are independent of the distance between subdomains so that the boundaries of more than two subdomains can touch at some points.
We also prove a local $W^{1,q}$-estimate for $W^{1,1}$-weak solutions in Corollary \ref{C1} by exploiting the argument in
 \cite{MR2465684,MR2548032}  combined with Theorem  \ref{M1}.

Second, we consider the stationary Navier-Stokes systems
 $$
 \begin{cases}
 \mathcal{L}u+\nabla p+u^\alpha D_\alpha u=D_{\alpha}f_{\alpha}&\quad\mbox{in }~\cD,\\
 \Div u=g&\quad\mbox{in }~\cD.
 \end{cases}
 $$
 We obtain any $W^{1,q}$-solution is Lipschitz and piecewise $C^1$, where $q\in[ d/2,\infty)$;  see Theorem \ref{M2} for the details. This result can be applied to $H^1$-weak solution to stationary Navier-Stokes systems with piecewise Dini mean oscillation coefficients in dimensions $d=2,3,4$.  Related work can be found in \cite{MR0520820}, in which the author considered the Laplace operator and proved the smoothness of every weak solution for $d=4$ provided the data are good enough.

Let us briefly describe our arguments based on Campanato's approach. Such approach was used in \cite{MR0717034,MR1184023} and further developed in \cite{MR2927619,MR3620893,MR3912724,MR3961984}. The key point is to show the mean oscillations of $Du$ and $p$ in balls vanish in a certain order as the radii of balls go to zero. Recalling the nature of the domain and the coefficients, $Du$ and $p$ are discontinuous in one direction, say, $x^d$, which is the main challenge in this paper. We overcome this difficulty by choosing a coordinate system according to the geometry of the subdomains and then using the weak type-$(1,1)$ estimates obtained in \cite[Lemma 3.4]{MR3912724} to control the $L^{1/2}$-mean oscillations of $D_{x'}u$ and the linear combinations $A^{d\beta}D_\beta u+pe_d-f_d$; see Lemma \ref{210127@lem1} for the details.  We point out that the proof in our case  is more involved than that in \cite{MR3912724} since our arguments and estimates depend on the coordinate system, and also more involved than that in \cite{MR3961984} because of  the pressure term and the divergence equation in the Stokes systems \eqref{stokes}. For example, in the proof of local boundedness of $Du$ and $p$ (see Step 4 in the proof of Theorem \ref{M1}), an additional difficulty appears from the pressure term on the right-hand side after the localization. For this, we adapt a delicate approximation argument and the fixed point theorem.

The rest of the paper is organized as follows. In Section \ref{Assumptions}, we fix our notation, introduce function spaces and assumptions on the domain, coefficients, and data, and then state our main results, Theorem \ref{M1} for stationary Stokes systems and Theorem \ref{M2} for stationary Navier-Stokes systems.
In Section \ref{S_3}, we provide the proofs of the main theorems.

\section{Assumptions and main results}\label{Assumptions}
We first fix some notation used throughout the paper.
We use $x=(x',x^d)$ to denote a generic point in the Euclidean space $\bR^d$, where $d\ge 2$ and $x'=(x^1,\ldots, x^{d-1})\in \bR^{d-1}$.
We also write $y=(y', y^d)$ and $x_0=(x_0',x_0^d)$, etc.
For $r>0$, we denote
$$B_{r}(x)=\{y\in\mathbb R^{d}: |y-x|<r\},\quad B'_{r}(x')=\{y'\in\mathbb R^{d-1}: |y'-x'|<r\}.$$
We often write $B_r$ and $B'_r$ instead of $B_r(0)$ and $B'_r(0)$, respectively.
For $k\in \{1,\dots,d\}$, we use $e_{k}$ to denote the $k$-th unit vector in $\mathbb R^{d}$.

Let $\Omega$ be a domain in $\bR^d$.
For $q\in (0, \infty]$, we define
$$
\tilde{L}^q(\Omega)=\{f\in L^q(\Omega): (f)_{\Omega}=0\},
$$
where $(f)_{\Omega}$ is the average of $f$ over $\Omega$, i.e.,
$$
(f)_\Omega=\fint_{\Omega} f\ dx=\frac{1}{|\Omega|}\int_{\Omega} f \ dx.
$$
For $q\in [1,\infty]$, we denote by $W^{1,q}(\Omega)$ the usual Sobolev space and by $W^{1,q}_0(\Omega)$ the completion of $C^\infty_0(\Omega)$ in $W^{1,q}(\Omega)$, where $C^\infty_0(\Omega)$ is the set of all infinitely differentiable functions with a compact support in $\Omega$.
We say that a function $\omega:[0,1]\to [0, \infty)$ is  a Dini function if it  is monotonically increasing and satisfies
$$
\int_0^{1}\frac{\omega(t)}{t}\,dt<+\infty.
$$
We also say that a function $f$ defined on $\Omega$ is  Dini continuous if the function $\varrho_f:[0, 1]\to [0, \infty)$ given by
$$
\varrho_f(t)=\sup_{\substack{x,y\in \Omega\\|x-y|\le t}} |f(x)-f(y)|
$$
is a Dini function.

\subsection{Assumptions on the domain}		\label{Sec_2_1}
Before we state our assumptions on the domain, we recall the definition of a domain having a $C^{1,\rm{Dini}}$ boundary.

\begin{definition}\label{def Dini}
Let $\Omega$ be a domain in $\bR^d$.
We say that $\Omega$ has a $C^{1,\text{Dini}}$ boundary if there exist a constant $R_0\in (0, 1]$ and a concave Dini function $\varrho_0$ such that the following holds.
For any $x_0=(x_0',x_0^d)\in \partial \Omega$, there exist a $C^1$ function  $\chi:\bR^{d-1}\to \bR$ and a coordinate system depending on $x_0$ such that
$$
\varrho_{\nabla_{x'}\chi}(t)\le \varrho_0(t) \quad \text{for all }\, t\in [0, R_0]
$$
and that in the new coordinate system, we have
$$
|\nabla_{x'}\chi(x_0')|=0
$$
and
\begin{equation}		\label{210218@B1}
\Omega\cap B_{R_0}(x_0)=\{x\in B_{R_0}(x_0): x^d>\chi(x')\}.
\end{equation}
\end{definition}

In this paper, we always assume that $\cD$ is a bounded domain in $\bR^d$ containing $M$ subdomains
$\sD_1,\ldots, \sD_M$ such that
\begin{enumerate}[i.]
\item
$\sD_M=\cD\setminus \big(\cup_{i=1}^{M-1}\overline{\sD_i}\big)$,
\item
for $i,j\in \{1,\ldots, M-1\}$ with $i\neq j$, we have either
\begin{equation}		\label{210219@eq3}
\overline{\sD_i}\subset \sD_j \quad \text{or}\quad \overline{\sD_i}\cap \overline{\sD_j}=\emptyset,
\end{equation}
\item
for $i\in \{1,\ldots M-1\}$, $\sD_i$ has a $C^{1,\rm{Dini}}$ boundary as in Definition \ref{def Dini} with the same constant $R_0$ and Dini function $\varrho_0$.
\end{enumerate}

Our assumptions on the domain, which look a bit different from those in \cite{MR3961984} are in fact identical.
Precisely, by disjointing the subdomains $\sD_1, \ldots, \sD_{M-1}$, one can understand $\cD$
 as a domain containing $M$ {\em{disjoint}} subdomains ${\cD}_1, \ldots, {\cD}_M$ such that
\begin{enumerate}[i${}^\prime$.]
\item
$\cD_M=\sD_M$.
\item
any point in $\cD$ belongs to the boundaries of at most two of the subdomains.
\item
for $i\in \{1,\ldots, M-1\}$, $\cD_i$ has a $C^{1, \rm{Dini}}$ boundary in an appropriate sense.
\end{enumerate}

Among the above two expressions of the nature of the domain,
the second  is useful to describe the regularity conditions on the coefficients and data, which  may have jump discontinuities across the interfacial boundaries; see Section \ref{Sec_2_2}.
On the other hand, the first expression is convenient to explain the regularity of the  boundaries by using Definition \ref{def Dini}.
Because the disjointed subdomains $\cD_i$ in the second expression may have ``narrow" regions, \eqref{210218@B1} is not guaranteed with the same constant $R_0$ independent of the distance between subdomains. For example,
if $M=3$, $\cD_1:=B_{1/2-\varepsilon}$, $\cD_2:=B_{1/2}\setminus \overline{B_{1/2-\varepsilon}}$, and $\cD_3:=B_1\setminus \overline{B_{1/2}}$, then when we explain the regularity of $\partial \cD_2$ via Definition \ref{def Dini}, we need to take $R_0$ to
 be less than $\varepsilon$ which is the distance between $\cD_{1}$ and $\cD_{3}$.
That is why we added ``appropriate sense" in the condition iii${}^\prime$.
In the following, we will use the notation $\cD_{i}$ introduced above to denote the subdomains.

We end this subsection with a remark that the condition \eqref{210219@eq3} can be relaxed to
$$
\sD_i\subset \sD_j \quad \text{or}\quad \sD_i\cap\sD_j=\emptyset,
$$
so that the boundaries of more than two subdomains touch at some points; see Remark \ref{210219@rmk2}.

\subsection{Assumptions on the coefficients and data}		\label{Sec_2_2}
We assume that the coefficients $A^{\alpha\beta}$ of the operator $\cL$ in \eqref{210322@eq1} are bounded and satisfy the strong ellipticity condition, that is, there exists $\nu\in (0,1)$ such that
\begin{equation}		\label{210210@eq4}
|A^{\alpha\beta}(x)|\le \nu^{-1}, \quad \sum_{\alpha,\beta=1}^d A^{\alpha\beta}(x)\xi_\beta\cdot \xi_\alpha\ge \nu \sum_{\alpha=1}^d |\xi_\alpha|^2
\end{equation}
for any $x\in \bR^d$ and $\xi_\alpha\in \bR^d$, $\alpha\in \{1,\ldots, d\}$.
We also assume that  the coefficients and data are of piecewise Dini mean oscillation satisfying Definition \ref{D1} below in the domain $\cD$ containing $M$ disjoint subdomains $\cD_1,\ldots, \cD_M$ as in Section \ref{Sec_2_1}.

\begin{definition}\label{D1}
Let $f\in L^1(\cD)$.
We say that $f$ is of piecewise Dini mean oscillation in $\cD$ if there exists a Dini function $\omega_f$ such that for any $x_0\in \cD$ and $r\in (0,1]$ satisfying $B_r(x_0)\subset \cD$, we have
\begin{equation}		\label{210129@eq1}
\fint_{B_r(x_0)}\big|f(x)-\hat{f}(x)\big|\,dx\le \omega_f(r),
\end{equation}
where $\hat{f}=\hat{f}_{x_0,r}$ is a piecewise continuous function on $B_r(x_0)$ given by
$$
\hat{f}(x)=\fint_{B_r(x_0)\cap {\cD}_i}f(y)\,dy \quad \text{if }\, x\in B_r(x_0)\cap {\cD}_i.
$$
\end{definition}

Our definition of a function of piecewise Dini mean oscillation is equivalent to the definition in \cite{MR3961984}, where the piecewise mean oscillation is measured by taking the infimum over the set of all piecewise constant functions.

\subsection{Main results}		\label{Sec_2_3}

The main results of this paper are as follows.

\begin{theorem}	\label{M1}
Let $\cD$ be a bounded domain in $\bR^d$ containing $M$ disjoint subdomains ${\cD}_1,\ldots, {\cD}_M$ with $C^{1, \rm{Dini}}$ boundaries as in Section \ref{Sec_2_1}.
Also, let $q\in (1,\infty)$  and
$(u, p)\in W^{1,q}(\cD)^d\times L^q(\cD)$ be a weak solution of
\begin{equation}\label{210201@eq4}
\begin{cases}
\mathcal{L}u+\nabla p=D_{\alpha}f_{\alpha}&\quad\mbox{in }~\cD,\\
\Div u=g&\quad\mbox{in }~\cD,
\end{cases}
\end{equation}
where $f_\alpha\in L^\infty(\cD)^d$ and $g\in L^\infty(\cD)$.
If $A^{\alpha\beta}$, $f_\alpha$, and $g$ are of piecewise Dini mean oscillation in $\cD$ satisfying Definition \ref{D1}, then for any $\cD'\Subset\cD$, we have
$$
(u, p)\in W^{1,\infty}(\cD')^d\times L^\infty(\cD')
$$
and
$$
(u, p)\in C^1\big(\overline{\cD_i}\cap \cD'\big)^d\times C\big(\overline{\cD_i}\cap \cD'\big), \quad i\in \{1,\ldots, M\}.
$$
If we further assume that
there exist $\gamma_0\in (0,1)$ and $K>0$ such that
\begin{equation}		\label{210215@eq3}
\varrho_0(r)\le K r^{\frac{\gamma_0}{1-\gamma_0}}, \quad \omega_{A^{\alpha\beta}}(r)+\omega_{f_\alpha}(r)+\omega_g(r)\le Kr^{\gamma_0}
\end{equation}
for all $r\in (0, R_0]$,
  then
$$
(u, p)\in C^{1,\gamma_0}\big(\overline{\cD_i}\cap \cD'\big)^d\times C^{\gamma_0}\big(\overline{\cD_i}\cap \cD'\big), \quad i\in \{1,\ldots, M\}.
$$
\end{theorem}

Related to the theorem above, we have a few remarks.

\begin{remark}		\label{210219@rmk2}
Upper bounds of the $L^\infty$-norms  and the modulus of continuity of $Du$ and $p$ can be found in the proof of the theorem; see Section \ref{S_3_1}.
Note that these upper bounds are independent of the distance between the subdomains.
Thus our results can be applied to the case when the boundaries of more than two  subdomains  touch at some points.

In the middle of the proof, we also proved that for any $x_0\in \cD'$, there exists a coordinate system associated with $x_0$ such that the certain linear combinations
$$
D_{x'} u \quad \text{and}\quad A^{d\beta}D_\beta u+p e_d-f_d
$$
are continuous at $x_0$.
Moreover, if \eqref{210215@eq3} holds, then they are H\"older continuous with the same exponent $\gamma_0$.
\end{remark}

\begin{remark}		\label{210215@rmk1}
The condition \eqref{210215@eq3} holds provided that the subdomains $\cD_i$ have $C^{1, \gamma_0/(1-\gamma_0)}$ boundaries and that $A^{\alpha\beta}$, $f_\alpha$, and $g$ are in $C^{\gamma_0}(\overline{\cD}_i)$ for each $i\in \{1,\ldots, M\}$.
\end{remark}

\begin{remark}		\label{210210@rmk1}
By the same reasoning as in \cite[Remark 2.4]{MR3912724},
one can extend the results in Theorem \ref{M1} to weak solutions of the system
$$
\begin{cases}
\mathcal{L}u+\nabla p=D_{\alpha}f_{\alpha}+f&\quad\mbox{in }~\cD,\\
\Div u=g&\quad\mbox{in }~\cD,
\end{cases}
$$
where $f\in L^{s}(\cD)^d$ with $s>d$.
The corresponding upper bounds of the $L^\infty$-norms and the modulus of continuity of $Du$ and $p$ can be found in Remark \ref{210219@rmk1} at the end of Section \ref{S_3_1}.
\end{remark}

In the corollary below, we present the $W^{1,q}$-estimate for $W^{1,1}$-weak solutions, which follows from Theorem \ref{M1}, the solvability results in \cite{MR3758532},  and the argument in Brezis \cite{MR2465684} (see also \cite[Appendix]{MR2548032}). One may refer to the proof of \cite[Theorem 2.5]{MR3912724}, where the authors proved the $W^{1,q}$-estimate for $W^{1,1}$-weak solutions to the Stokes system with partially Dini mean oscillation coefficients.

\begin{corollary}\label{C1}
Let $\cD$ be a bounded domain in $\bR^d$ containing $M$ disjoint subdomains ${\cD}_1,\ldots, {\cD}_M$ as in Section \ref{Sec_2_1}.
Also, let $(u, p)\in W^{1,1}(\cD)^d\times L^1(\cD)$ be a weak solution of \eqref{210201@eq4}, where $f_\alpha\in L^{q}(\cD)^d$ and $g\in L^{q}(\cD)$ with $q\in (1, \infty)$.
If $A^{\alpha,\beta}$, $f_\alpha$, and $g$ are piecewise Dini mean oscillation in $\cD$ satisfying Definition \ref{D1}, then  for $\cD'\Subset \cD$,  we have
$$
(u, p)\in W^{1,q}(\cD')^d\times L^{q}(\cD')
$$
with the estimate
$$
\|u\|_{W^{1,q}(\cD')}+\|p\|_{L^q(\cD')}\le N \big(\|u\|_{W^{1,1}(\cD)}+\|p\|_{L^1(\cD)}+\|f_\alpha\|_{L^q(\cD)}+\|g\|_{L^q(\cD)}\big),
$$
where the constant $N$ depends only on $d$, $\nu$, $M$, $R_0$, $\varrho_0$, $\omega_{A^{\alpha\beta}}$, $q$, and $\operatorname{dist}(\partial \cD,  \cD')$.
\end{corollary}

\begin{remark}
From Corollary \ref{C1}, the results in Theorem \ref{M1} still hold under the assumption that $(u, p)\in W^{1,1}(\cD)^d\times L^1(\cD)$.
\end{remark}

We also consider stationary Navier-Stokes systems with piecewise Dini mean oscillation coefficients.

\begin{theorem}	\label{M2}
Let $\cD$ be a bounded domain in $\bR^d$ containing $M$ disjoint subdomains ${\cD}_1,\ldots, {\cD}_M$ with $C^{1, \rm{Dini}}$ boundaries as in Section \ref{Sec_2_1}.
Also, let $q\in (1, \infty)$ with $q\ge d/2$ and
 $(u, p)\in W^{1,q}(\cD)^d\times L^q(\cD)$ be a weak solution of
$$
\begin{cases}
\mathcal{L}u+\nabla p+u^\alpha D_\alpha u=D_{\alpha}f_{\alpha}&\quad\mbox{in }~\cD,\\
\Div u=g&\quad\mbox{in }~\cD,
\end{cases}
$$
where $f_\alpha\in L^\infty(\cD)^d$ and $g\in L^\infty(\cD)$.
If $A^{\alpha\beta}$, $f_\alpha$, and $g$ are of piecewise Dini mean oscillation in $\cD$ satisfying Definition \ref{D1}, then for any $\cD'\Subset\cD$, we have
$$
(u, p)\in W^{1,\infty}(\cD')^d\times L^\infty(\cD')
$$
and
$$
(u, p)\in C^1\big(\overline{\cD_i}\cap \cD'\big)^d\times C\big(\overline{\cD_i}\cap \cD'\big), \quad i\in \{1,\ldots, M\}.
$$
If we further assume \eqref{210215@eq3}, then
$$
(u, p)\in C^{1,\gamma_0}\big(\overline{\cD_i}\cap \cD'\big)^d\times C^{\gamma_0}\big(\overline{\cD_i}\cap \cD'\big), \quad i\in \{1,\ldots, M\}.
$$
\end{theorem}

As a corollary of Theorem \ref{M2}, when $d=2,3,4$, any $H^1$-weak solution to the stationary Navier-Stokes system with piecewise Dini mean oscillation coefficients is Lipschitz.

\section{Proofs of main theorems}\label{S_3}
Throughout this paper, we use the following notation.

\begin{notation}
For nonnegative (variable) quantities $A$ and $B$, we denote $A \lesssim B$ if there exists a generic positive constant $C$ such that $A \le CB$.
We add subscript letters like $A\lesssim_{a,b} B$ to indicate the dependence of the implicit constant $C$ on the parameters $a$ and $b$.
\end{notation}

\subsection{Proof of Theorem \ref{M1}}		\label{S_3_1}
We begin the proof with the following observation.
Under the assumptions on the domain $\cD$ with a scaling whose parameter depends only on $d$, $R_0$, $\varrho_0$, and $\operatorname{dist}(\partial \cD, \cD')$,
we may suppose that for any $x_0\in \cD'$,  there exist $C^{1,\rm{Dini}}$ functions $\chi_i:\bR^{d-1}\to \bR$, $i\in \{1,\ldots, \ell\}$ for some $\ell<M$, and a coordinate system such that the following properties hold in the new coordinate system (called the coordinate system associated with $x_0$):
\begin{enumerate}[$({\bf{A}}1)$]
\item
We have that
$$
\varrho_{\nabla_{x'}\chi_i}(r)\le \varrho_0(r)
$$
for all $r\in [0, R_0]$ and $i\in \{1,\ldots, \ell\}$, and that
$$
\chi_0(x')< \chi_1(x')< \cdots < \chi_{\ell}(x')<\chi_{l+1}(x')
$$
for all  $x'\in B_1'(x_0)$, where we have adopted the notation $\chi_0\equiv x_0^d-1$ and $\chi_{l+1}\equiv x_0^d+1$.
\item
$B_1(x_0)\subset \cD$ and $B_1(x_0)$ is divided into $\ell+1$ disjoint subdomains
$$
\widehat{\cD}_i:=\{x\in B_1(x_0):\chi_{i-1}(x')<x^d<\chi_i(x')\}, \quad i\in \{1, \ldots, \ell+1\}.
$$
Here, in an appropriate sense one may think of $\widehat{\cD}_i$ as ${\cD}_i\cap B_1(x_0)$.
Moreover,
$$
x_0\in \widehat{\cD}_{i_0} \cup \partial \widehat{\cD}_{i_0}\quad \text{for some }\, i_0\in \{1, \ldots, \ell+1\},
$$
the closest point on $\partial \widehat{\cD}_{i_0}$ to $x_0$ is $(x_0', \chi_{i_0}(x_0'))$, and $\nabla_{x'} \chi_{i_0}(x_0')=0'$.
\end{enumerate}

Throughout this proof, we shall use the following notation and properties in the coordinate system associated with $x_0$ satisfying $({\bf{A}}1)$ and $({\bf{A}}2)$.

\begin{enumerate}[$({\bf{B}}1)$]
\item
For $i\in \{1,\ldots, \ell+1\}$, we denote
$$
\Omega_i=\{x\in B_{1}(x_0): \chi_{i-1}(x_0')<x^d<\chi_i(x'_0)\}.
$$
By \cite[Lemma 2.3]{MR3961984}, there exists $R_1=R_1(R_0, \varrho_0)\in (0, R_0]$ such that for any $r\in (0, R_1]$,
\begin{equation}		\label{210128@eq1}
r^{-d}|(\widehat{\cD}_i \setminus \Omega_i)\cap B_{r}(x_{0})|\lesssim_{d, M, \varrho_0} \varrho_1(r),
\end{equation}
where $\varrho_1$ is a Dini function derived from $\varrho_0$.
\item
Let $f$ be of piecewise Dini mean oscillation in $\cD$ satisfying Definition \ref{D1} with a Dini function $\omega_f$.
For $r\in (0, R_1]$, we define piecewise continuous functions $\hat{f}=\hat{f}_{x_0, r}$ and $\bar{f}=\bar{f}_{x_0, r}$ in $B_r(x_0)$ by
$$
\hat{f}(x)=\fint_{\widehat{\cD}_i\cap B_r(x_0)} f(y)\ dy \quad \text{if }\, x\in B_r(x_0)\cap \widehat{\cD}_i
$$
and
\begin{equation}\label{210129@eq2a}
\bar{f}(x)=\fint_{\widehat{\cD}_i\cap B_r(x_0)} f(y)\ dy \quad \text{if }\, x\in B_r(x_0)\cap \Omega_i,
\end{equation}
where $\bar{f}$ is indeed a function of $x^d$.
Since
$\hat{f}\equiv \bar{f}$ in $B_r(x_0)\cap \widehat{\cD}_i\cap \Omega_i$,
 by \eqref{210128@eq1}, we have
$$
\begin{aligned}
\fint_{B_r(x_0)}|\hat{f}-\bar{f}|\ dx
&=\frac{1}{|B_r|}\sum_{i=1}^{\ell+1}\int_{(\widehat{\cD}_i \setminus \Omega_i)\cap B_r(x_0)}|\hat{f}-\bar{f}|\,dx\\
&\lesssim \|f\|_{L^\infty(B_r(x_0))}\varrho_1(r).
\end{aligned}
$$
From this together with \eqref{210129@eq1}, it follows that
\begin{equation}		\label{210129@eq2}
\fint_{B_r(x_0)} |f-\bar{f}|\,dx\lesssim_{d, M, \varrho_0} \omega_f(r)+\|f\|_{L^\infty(B_r(x_0))}\varrho_1(r).
\end{equation}
\item
We set
$$
U=A^{d\beta}D_\beta u+pe_d-f_d.
$$
For $y\in \cD$ and $r>0$ with $B_r(y)\subset B_1(x_0)$,
we define
$$
\Phi_{x_0}(y,r)=\inf_{\Theta\in\mathbb R^{d\times d}}\bigg(\fint_{B_{r}(y)}|(D_{x'}u, U)-\Theta|^{\frac{1}{2}}\ dx\bigg)^{2},
$$
where we used the subindex $x_0$ to indicate that the function is defined in the coordinate system associated with $x_0$.
\end{enumerate}

To prove Theorem \ref{M1}, we will use the following decay estimates.

\begin{lemma}\label{210127@lem1}
Let $x_0\in \cD'$, $r\in (0, R_1]$,  and $\gamma\in (0,1)$.
Then under the same hypotheses of Theorem \ref{M1} with an additional assumption that $Du$ and $p$ are locally bounded, there exists $N=N(d, \nu, M, \varrho_0, \gamma)>0$ such that the following assertions hold.
\begin{enumerate}[$(i)$]
\item
For any $\rho\in (0, r]$, we have
\begin{equation}		\label{210225@C1}
\begin{aligned}
\Phi_{x_0}(x_0,\rho)
&\leq N\Big(\frac{\rho}{r}\Big)^{\gamma}\Phi_{x_0}(x_0, r)+N\|Du\|_{L^{\infty}(B_{r}(x_{0}))}\big(\tilde{\omega}_{A^{\alpha\beta}}(\rho)+\tilde{\varrho}_1(\rho)\big)\\
&\quad +N\big(\|f_\alpha\|_{L^\infty(B_{r}(x_0))}+\|g\|_{L^\infty(B_{r}(x_0))}\big)\tilde{\varrho}_1(\rho)\\
&\quad +N\big(\tilde{\omega}_{f_{\alpha}}(\rho)+\tilde{\omega}_{g}(\rho)\big)
\end{aligned}
\end{equation}
\item
For any $y\in B_{r/2}(x_0)$ and $\rho\in (0, r/2]$ such that  $B_\rho(y)\subset \widehat{\cD}_{i_1}$ for some $i_1\in \{1, \ldots, \ell+1\}$, we have
\begin{equation}		\label{210225@C2}
\begin{aligned}
\Phi_{x_0}(y, \rho)
&\le N \bigg(\frac{\rho}{r}\bigg)^\gamma \Phi_{y}(y, r/2)+N\|Du\|_{L^\infty(B_{r/2}(y))}\big(\tilde{\omega}_{A^{\alpha\beta}}(\rho)+\tilde{\varrho}_1(\rho)\big)\\
&\quad+N\big(\|f_\alpha\|_{L^\infty(B_{r/2}(y))}+\|g\|_{L^\infty(B_{r/2}(y))}\big)\big(\tilde{\omega}_{A^{\alpha\beta}}(\rho)+\tilde{\varrho}_1(\rho)\big)\\
&\quad +N\big(\tilde{\omega}_{f_{\alpha}}(\rho)+\tilde{\omega}_{g}(\rho)\big).
\end{aligned}
\end{equation}
\end{enumerate}
In the above, $\tilde\omega_{\bullet}$ and $\tilde{\varrho}_1$ are Dini functions derived from $\omega_{\bullet}$ and $\varrho_1$, respectively, as formulated in \eqref{210129@eq6}.
\end{lemma}

\begin{proof}
We may assume that $x_0=0$ for simplicity of notation.
For a given function $f$, we denote by $\bar{f}=\bar{f}(x^d)$ the piecewise constant function in $B_r$ defined as in \eqref{210129@eq2a}.

We first prove the assertion $(i)$.
Let $\cL_0$ be an elliptic operator given by
$$
\cL_0 u=D_{\alpha}(\bar{A}^{\alpha\beta}D_{\beta}u),
$$
and set
$$
u_{e}=u-\int_{-1}^{x^{d}}u_{0}\ ds,\quad p_{e}=p-p_{0},
$$
where $u_{0}=(u_{0}^{1},\dots,u_{0}^{d})^{\top}$ and $p_{0}$ are functions of $x^d$  satisfying
$$
u^d_0=\bar{g}, \quad \bar{A}^{dd}u_{0}+p_{0}e_d=\bar{f}_{d}.
$$
Then $(u_e, p_e)$ satisfies
\begin{align*}
\begin{cases}
\cL_0u_{e}+\nabla p_e=D_{\alpha}F_{\alpha}&\quad\mbox{in~}~B_{r},\\
\Div u_{e}=G&\quad\mbox{in~}~B_{r},
\end{cases}
\end{align*}
where $F_{\alpha}=(\bar{A}^{\alpha\beta}-A^{\alpha\beta})D_{\beta}u+f_{\alpha}-\bar{f}_{\alpha}$ and $G=g-\bar{g}$.
We decompose
\begin{equation}\label{decom}
(u_{e},p_{e})=(v,p_{1})+(w,p_{2}),
\end{equation}
where $(v,p_{1})\in W_{0}^{1,2}(B_{r})^{d}\times \tilde L^{2}(B_{r})$ is a unique weak solution of
$$
\begin{cases}
\cL_0 v+\nabla p_1=D_{\alpha}(I_{B_{r/4}}F_{\alpha})&\quad\mbox{in~}~B_{r},\\
\Div v=I_{B_{r/4}}G-(I_{B_{r/4}}G)_{B_{r}}&\quad\mbox{in~}~ B_{r}.
\end{cases}
$$
Here, $I_{B_{r/4}}$ is the characteristic function.
Then by \cite[Lemma 3.4]{MR3912724} with scaling and relabeling the coordinate axes, we have for all $t>0$ that
$$
\big|\{x\in B_{r/4}: |Dv(x)|+|p_{1}(x)|>t\}\big|\lesssim_{d,\nu}\frac{1}{t}\int_{B_{r/4}}\big(|F_{\alpha}|+|G|\big)\ dx.
$$
This implies that (c.f. \cite[Eq. (4.5)]{MR3912724})
\begin{equation}		\label{210129@eq3}
\bigg(\fint_{B_{r/4}} (|Dv|+|p_1|)^{\frac{1}{2}} \ dx\bigg)^2\lesssim  \fint_{B_{r/4}} (|F_\alpha|+|G|)\ dx.
\end{equation}
On the other hand, since  $(w, p_2)$ satisfies
\begin{align*}
\begin{cases}
\cL_0 w+\nabla p_2=0&\quad\mbox{in~}~B_{r/4},\\
\Div w=(I_{B_{r/4}}G)_{B_{r}}&\quad\mbox{in~}~ B_{r/4},
\end{cases}
\end{align*}
by \cite[Eq. (3.7)]{MR3912724}, we obtain
\begin{equation}\label{Holder L0 w}
\begin{aligned}
&\bigg(\fint_{B_{\kappa r}}\big|D_{x'}w-(D_{x'}w)_{B_{\kappa r}}\big|^{\frac{1}{2}}+\big|W-(W)_{B_{\kappa r}}\big|^{\frac{1}{2}}\ dx\bigg)^{2}\\
&\lesssim \kappa \inf_{\Theta\in\mathbb R^{d\times d}}\bigg(\fint_{B_{r/4}}|(D_{x'}w,W)-\Theta|^{\frac{1}{2}}\ dx\bigg)^{2}
\end{aligned}
\end{equation}
for any $\kappa\in(0,1/8]$, where $W=\bar{A}^{d\beta}D_{\beta}w+p_{2}e_{d}$.

Now we set
$$
U_{e}=\bar{A}^{d\beta}D_{\beta}u_{e}+p_{e}e_{d},
$$
and observe that
\begin{equation}		\label{210129@eq4}
D_{x'}u_e=D_{x'}u, \quad U-U_e=(A^{d\beta}-\bar{A}^{d\beta})D_\beta u-(f_d-\bar{f}_d).
\end{equation}
By \eqref{decom}--\eqref{Holder L0 w} and the triangle inequality, we have
\begin{align*}
&\bigg(\fint_{B_{\kappa r}}\big|D_{x'}u_{e}-(D_{x'}w)_{B_{\kappa r}}\big|^{\frac{1}{2}}+\big|U_{e}-(W)_{B_{\kappa r}}\big|^{\frac{1}{2}}\ dx\bigg)^{2}\nonumber\\
&\lesssim \bigg(\fint_{B_{\kappa r}}\big|D_{x'}w-(D_{x'}w)_{B_{\kappa r}}\big|^{\frac{1}{2}}+\big|W-(W)_{B_{\kappa r}}\big|^{\frac{1}{2}}\ dx\bigg)^{2}\nonumber\\
&\quad +\bigg(\fint_{B_{\kappa r}}(|Dv|+|p_{1}|)^{\frac{1}{2}}\ dx\bigg)^{2}\nonumber\\
&\lesssim \kappa \inf_{\Theta\in\mathbb R^{d\times d}}\bigg(\fint_{B_{r/4}}|(D_{x'}w,W)-\Theta|^{\frac{1}{2}}\ dx\bigg)^{2}+\kappa^{-2d}\fint_{B_{r/4}}(|F_\alpha|+|G|)\ dx
\\
&
\lesssim \kappa \inf_{\Theta\in \bR^{d\times d}}\bigg(\fint_{B_{r}}|(D_{x'}u_e, U_e)-\Theta|^{\frac{1}{2}}\ dx\bigg)^{2}+\kappa^{-2d}\fint_{B_{r}}(|F_\alpha|+|G|)\ dx.
\end{align*}
From this together with \eqref{210129@eq2} and \eqref{210129@eq4}, we get
$$
\begin{aligned}
\Phi_0(0, \kappa r)
&\le N_0\kappa  \Phi_0(0, r)+ N_0 \kappa^{-2d} \|Du\|_{L^\infty(B_{r})}(\omega_{A^{\alpha\beta}}(r)+\varrho_1(r))\\
&\quad + N_0\kappa^{-2d}\big(\|f_\alpha\|_{L^\infty(B_r)}+\|g\|_{L^\infty(B_r)}\big)\varrho_1(r) +N_0\kappa^{-2d}(\omega_{f_\alpha}(r)+\omega_g(r)),
\end{aligned}
$$
where $N_0=N_0(d, \nu,M, \varrho_0)>0$.
Fix $\kappa\in (0, 1/8]$ small enough so that $N_0\kappa^{1-\gamma}\le 1$.
Then,
$$
\begin{aligned}
\Phi_0(0, \kappa r)
&\le \kappa^{\gamma}\Phi_0(0, r)+N\|Du\|_{L^\infty(B_r)}(\omega_{A^{\alpha\beta}}(r)+\varrho_1(r))\\
&\quad + N\big(\|f_\alpha\|_{L^\infty(B_r)}+\|g\|_{L^\infty(B_r)}\big)\varrho_1(r)+N(\omega_{f_\alpha}(r)+\omega_g(r)),
\end{aligned}
$$
where $N=N(d, \nu, M, \varrho_0, \gamma)>0$.
Let  $\tilde{\omega}_{\bullet}$ and $\tilde{\varrho}_0$ be Dini functions defined by
\begin{equation}		\label{210129@eq6}
\begin{aligned}
\tilde{\omega}_{\bullet}(r)&=\sum_{i=1}^\infty \kappa^{\gamma i}\big(\omega_{\bullet}(\kappa^{-i}r)[\kappa^{-i}r<1]+\omega_{\bullet}(1)[\kappa^{-i}r\ge 1]\big),\\
\tilde{\varrho}_{1}(r)&=\sum_{i=1}^\infty \kappa^{\gamma i}\big(\varrho_1(\kappa^{-i}r)[\kappa^{-i}r<1]+\varrho_1(1)[\kappa^{-i}r\ge 1]\big),
\end{aligned}
\end{equation}
where we used the Iverson bracket notation, i.e., $[P] = 1$ if $P$ is true and $[P] = 0$ otherwise.
By iterating and using the fact that
$$
\sum_{i=1}^j \kappa^{\gamma(i-1)}\omega_{\bullet}(\kappa^{j-i}r)\le \kappa^{-\gamma}\tilde{\omega}_{\bullet}(\kappa^jr), \quad j\in \{1,2,\ldots\},
$$
we obtain
\begin{equation}		\label{210201@eq1}
\begin{aligned}
&\Phi_0(0, \kappa^j r)
\le \kappa^{\gamma j}\Phi_0(0, r)+N\|Du\|_{L^\infty(B_r)}\big(\tilde{\omega}_{A^{\alpha\beta}}(\kappa^jr)+\tilde{\varrho}_1(\kappa^jr)\big)\\
&\quad +N\big(\|f_\alpha\|_{L^\infty(B_r)}+\|g\|_{L^\infty(B_r)}\big)\tilde{\varrho}_1(\kappa^jr)+N\big(\tilde{\omega}_{f_\alpha}(\kappa^jr)+\tilde{\omega}_g(\kappa^jr)\big),
\end{aligned}
\end{equation}
which also obviously holds for $j=0$.
Finally,  for $\rho\in (0, r]$, by taking the nonnegative integer  $j$ such that
$\kappa^{j+1}<\rho/r\le \kappa^{j}$ and using \eqref{210201@eq1} with $\rho$ in place of $\kappa^j r$, we get the desired estimate.

Next, we prove the assertion $(ii)$.
For a given function $f$, we define
$$
\hat{f}=\fint_{B_\rho(y)} f(x)\,dx.
$$
Notice from the definition of $U$ that for any $\Theta_\beta\in \bR^d$ and $\theta\in \bR$, we have
$$
|U-\Theta_0|^{\frac{1}{2}}\le \big|(A^{d\beta}-\hat{A}^{d\beta})D_\beta u\big|^{\frac{1}{2}}+\big|\hat{A}^{d\beta} (D_\beta u-\Theta_\beta)\big|^{\frac{1}{2}}+|p-\theta|^{\frac{1}{2}}+|f_d-\hat{f}_d|^{\frac{1}{2}},
$$
where $\Theta_0=\hat{A}^{d\beta}\Theta_\beta+\theta e_d-\hat{f}_d$, in the coordinate system associated with $x_0$.
By averaging the above inequality on $B_\rho(y)$, taking the square,  and using \eqref{210129@eq1} (with the fact that $B_\rho(y)$ is contained in a subdomain), we obtain
$$
\begin{aligned}
\bigg(\fint_{B_\rho(y)}|U-\Theta_0|^\frac{1}{2}\,dx\bigg)^2
&\lesssim \bigg(\fint_{B_{\rho}(y)}|D_\beta u-{\Theta_\beta }|^{\frac{1}{2}}+|p-\theta|^{\frac{1}{2}}\ dx\bigg)^2\\
&\quad +\|Du\|_{L^\infty(B_{\rho}(y))}\omega_{A^{\alpha\beta}}(\rho)+\omega_{f_\alpha}(\rho).
\end{aligned}
$$
From this we get
\begin{equation}		\label{210225@A1}
\Phi_{x_0}(y, \rho)\lesssim \Psi(y,\rho)+\|Du\|_{L^\infty(B_{\rho}(y))}\omega_{A^{\alpha\beta}}(\rho)+\omega_{f_\alpha}(\rho),
\end{equation}
where
$$
\Psi(y, \rho):=\inf_{\substack{\theta\in \bR\\ \Theta\in \bR^{d\times d}}}\bigg(\fint_{B_{\rho}(y)}|Du-{\Theta}|^{\frac{1}{2}}+|p-\theta|^{\frac{1}{2}}\ dx\bigg)^2.
$$
Note that $\Psi(y,\rho)$ is independent of coordinate systems.

We now control the quantity $\Psi(y, \rho)$ in the coordinate system associated with $y$.
Using \eqref{210129@eq1} and the relation
\begin{equation}		\label{210323@EQ2}
D_du^d=g-\sum_{i=1}^{d-1}D_i u^i,
\end{equation}
we have
\begin{align}
\nonumber
&\inf_{\theta\in \bR} \bigg(\fint_{B_{\rho}(y)}|D_d u^d-{\theta}|^{\frac{1}{2}}\ dx\bigg)^2\\
\nonumber
&\lesssim \sum_{i=1}^{d-1}\inf_{{\theta}\in \bR} \bigg(\fint_{B_{\rho}(y)}|D_iu^i-{\theta}|^{\frac{1}{2}}\ dx\bigg)^2+\fint_{B_{\rho}(y)}|g-\hat{g}| \ dx\\
\label{210204@A1}
&\lesssim \Phi_{y} (y, \rho)+\omega_g(\rho).
\end{align}
Note that
\begin{equation}		\label{210420@A1}
\sum_{j=1}^{d-1}A^{dd}_{ij}D_du^j=U^i-\sum_{j=1}^d \sum_{\beta=1}^{d-1}A^{d\beta}_{ij}D_\beta u^j-A^{dd}_{id}D_du^d+f_d^i, \quad i\in \{1,\ldots, d-1\},
\end{equation}
where, by the ellipticity condition on $A^{\alpha\beta}$, $(A^{dd}_{ij})^{d-1}_{i,j=1}$ is nondegenerate.
Hence,
$$
\cX=\cY \cZ,
$$
where
$$
\cX=(D_du^1,\ldots,D_d u^{d-1})^{\top}, \quad \cY=\big((A^{dd}_{ij})^{d-1}_{i,j=1}\big)^{-1},
$$
$$
\cZ=(\cZ^1,\ldots, \cZ^{d-1})^{\top},  \quad \cZ^i=U^i-\sum_{j=1}^d \sum_{\beta=1}^{d-1}A^{d\beta}_{ij}D_\beta u^j-A^{dd}_{id}D_du^d+f_d^i.
$$
Since
$$
|\cX-\hat{\cY}{\vartheta}|\le |(\cY-\hat{\cY})\cZ|+|\hat{\cY}(\cZ-{\vartheta})|, \quad \forall ~\vartheta\in \bR^{d-1},
$$
we see that
$$
\begin{aligned}
&\inf_{{\vartheta}\in \bR^{d-1}} \bigg(\fint_{B_{\rho}(y)}|\cX-{\vartheta}|^{\frac{1}{2}}\ dx\bigg)^2\\
&\lesssim \bigg(\fint_{B_{\rho}(y)}|\cY-\hat{\cY}|^\frac{1}{2}\ dx\bigg)^{2}\|\cZ\|_{L^\infty(B_{\rho}(y))}+\inf_{{\vartheta}\in \bR^{d-1}}\bigg(\fint_{B_{\rho}(y)} |\cZ-{\vartheta}|^{\frac{1}{2}}\ dx\bigg)^2\\
&\lesssim \Phi_{y}(y, \rho) +\big(\|Du\|_{L^\infty(B_\rho(y))}+\|f_\alpha\|_{L^\infty(B_{\rho}(y))}\big)\omega_{A^{\alpha\beta}}(\rho) +\omega_{f_\alpha}(\rho)+\omega_g(\rho)=:K_0.
\end{aligned}
$$
From this together with \eqref{210204@A1}, we get
$$
\inf_{{\Theta}\in \bR^{d\times d}}\bigg(\fint_{B_{\rho}(y)}|Du-{\Theta}|^{\frac{1}{2}}\ dx\bigg)^2\lesssim K_0.
$$
By the relation
\begin{equation}		\label{210323@EQ3}
p=U^d-\sum_{j=1}^d\sum_{\beta=1}^{d}A^{d\beta}_{dj}D_\beta u^j+f_d^d,
\end{equation}
we also have
$$
\inf_{{\theta}\in \bR}\bigg(\fint_{B_{\rho}(y)}|p-{\theta}|^{\frac{1}{2}}\ dx\bigg)^2\lesssim K_0.
$$
Combining these inequalities, we obtain that
$\Psi(y, \rho)\lesssim K_0$,
which together with \eqref{210225@A1} gives $\Phi_{x_0}(y, \rho)\lesssim K_0$.
We finish the proof of the assertion $(ii)$ by applying \eqref{210225@C1} with $y$ and $r/2$ in place of $x_0$ and $r$, to bound  $K_0$ by the right-hand side of \eqref{210225@C2}.
\end{proof}

We are ready to prove Theorem \ref{M1}.

\begin{proof}[Proof of Theorem \ref{M1}]
We adapt the arguments in the proof of \cite[Theorem 1.1]{MR3961984}.
Let  $\tilde\omega_{\bullet}$ and $\tilde{\varrho}_1$ be the Dini functions derived from $\omega_{\bullet}$ and $\varrho_1$, respectively, as formulated in \eqref{210129@eq6} with a fixed $\gamma\in (0,1)$.
We denote
$$
\cF(r)=\int_0^r\frac{\tilde{\omega}_{f_{\alpha}}(t)+\tilde{\omega}_{g}(t)}{t}\ dt.
$$
For given $y\in \cD$ and $\rho>0$ with $B_\rho(y)\subset B_1(x_0)$, we let $\Theta_{x_0}(y, \rho)\in \bR^{d\times d}$ be such that
$$
\Phi_{x_0}(y,\rho)=\bigg(\fint_{B_\rho(y)}|(D_{x'}u, U)-\Theta_{x_0}(y, \rho)|^{\frac{1}{2}}\ dx\bigg)^{2}.
$$
We divide the proof into four steps.
In the first step, we will derive an a priori $L^\infty$-estimate for $(Du, p)$ under the assumption that $(Du, p)$ is locally bounded.
We then obtain an estimate of the modulus of continuity of $(D_{x'}u, U)$ in the second step, from which the  piecewise continuity of $(Du, p)$ follows.
In the third step, we shall derive an a priori estimate of the modulus of continuity of $(Du, p)$ under the additional condition \eqref{210215@eq3}.
In the last step, we shall show that $(Du, p)$ is indeed locally bounded by using the technique of flattening the boundary and a fixed point argument combined with partial Schauder estimates for Stokes systems.

{\bf{Step 1}}.
Let $r\in (0, R_1]$.
Note that   Lemma \ref{210127@lem1} $(i)$ implies
$$
\lim_{i\to \infty} \Phi_{x_0}(x_0, \kappa^ir)=0
$$
for all $x_0\in \cD'$, where $\kappa\in (0,1/8]$ is the constant from the proof of Lemma \ref{210127@lem1}.
Thus, using the assumption that $Du$ and $p$ are bounded, we have
$$
\lim_{i \to \infty} \Theta_{x_0}(x_0, \kappa^{i}r)=(D_{x'}u(x_0), U(x_0))
$$
for a.e. $x_0\in \cD'$, in the coordinate systems associated with $x_0$ satisfying  $({\bf{A}}1)$ and $({\bf{A}}2)$.
By the same iteration argument that led to  \cite[Eq. (4.10)]{MR3912724}, we have
\begin{equation}\label{210226@A1}
|(D_{x'}u(x_0), U(x_0))-\Theta_{x_0}(x_0, r)|\lesssim \sum_{i=0}^\infty \Phi_{x_0} (x_0, \kappa^i r).
\end{equation}
Since
$$
|\Theta_{x_0}(x_0, r)|\lesssim r^{-d} \big(\|D_{x'}u\|_{L^1(B_{r}(x_0))}+\|U\|_{L^1(B_{r}(x_0))}\big),
$$
by Lemma \ref{210127@lem1} $(i)$ and the fact that
\begin{equation}		\label{210226@A2}
\sum_{i=0}^\infty \tilde{\omega}_{\bullet}(\kappa^i r)\lesssim \int_0^{r}\frac{\tilde{\omega}_{\bullet}(t)}{t}\ dt, \quad \sum_{i=0}^\infty \tilde{\varrho}_{1}(\kappa^i r)\lesssim \int_0^{r}\frac{\tilde{\varrho}_{1}(t)}{t}\ dt,
\end{equation}
we obtain
$$
\begin{aligned}
|D_{x'}u(x_0)|+|U(x_0)|&\lesssim_{d, \nu, M, \varrho_0, \gamma} \|Du\|_{L^\infty(B_{r}(x_0))}\int_0^{r} \frac{\tilde{\omega}_{A^{\alpha\beta}}(t)+\tilde{\varrho}_1(t)}{t}\ dt\\
&\quad +r^{-d} \big(\|D_{x'}u\|_{L^1(B_{r}(x_0))}+\|U\|_{L^1(B_{r}(x_0))}\big)\\
&\quad +\big(\|f_\alpha\|_{L^\infty(B_{r}(x_0))}+\|g\|_{L^\infty(B_{r}(x_0))}\big)\int_0^{r} \frac{\tilde{\varrho}_1(t)}{t}\ dt+\cF(r).
\end{aligned}
$$
From this together with the fact that
$$
|Du|+|p|\lesssim_{d, \nu} |D_{x'}u|+|U|+|f_d|+|g|,
$$
we get
$$
\begin{aligned}
&|Du(x_0)|+|p(x_0)|\le N_0\|Du\|_{L^\infty(B_{r}(x_0))}\int_0^{r} \frac{\tilde{\omega}_{A^{\alpha\beta}}(t)+\tilde{\varrho}_1(t)}{t}\ dt\\
&\quad +N_0r^{-d} \big(\|Du\|_{L^1(B_{r}(x_0))}+\|p\|_{L^1(B_{r}(x_0))}\big)\\
&\quad +N_0\big(\|f_\alpha\|_{L^\infty(B_{r}(x_0))}+\|g\|_{L^\infty(B_{r}(x_0))}\big)\bigg(1+\int_0^{r} \frac{\tilde{\varrho}_1(t)}{t}\ dt\bigg)+N_0\cF(r),
\end{aligned}
$$
where $N_0=N_0(d, \nu, M, \varrho_0, \gamma)$.
Taking $r_0\in (0, R_1]$ sufficiently small so that
$$
N_0\int_0^{r_0}\frac{\tilde{\omega}_{A^{\alpha\beta}}(t)+\tilde{\varrho}_1(t)}{t}\ dt\le \frac{1}{3^d},
$$
we have
$$
\begin{aligned}
|Du(x_0)|+|p(x_0)|&\le 3^{-d}\|Du\|_{L^\infty(B_r(x_0))}\\
&\quad +N_0r^{-d} \big(\|Du\|_{L^1(B_r(x_0))}+\|p\|_{L^1(B_r(x_0))}\big)\\
&\quad +N_0\big(\|f_\alpha\|_{L^\infty(B_r(x_0))}+\|g\|_{L^\infty(B_r(x_0))}\big)+N_0\cF(r)
\end{aligned}
$$
for all $r\in (0, r_0]$.
Note that the above inequality holds for a.e. $x_0\in \cD'$ and does not depend on coordinate systems.
Therefore, by the same iteration argument that led to \cite[Eq. (4.16)]{MR3912724}, we obtain the following $L^\infty$-estimate for $Du$ and $p$:
\begin{equation}		\label{210202@eq1}
\begin{aligned}
&\|Du\|_{L^\infty(B_{r/2}(x_0))}+\|p\|_{L^\infty(B_{r/2}(x_0))}\\
&\le N r^{-d} \big(\|Du\|_{L^1(B_r(x_0))}+\|p\|_{L^1(B_r(x_0))}\big)\\
&\quad +N \big(\|f_\alpha\|_{L^\infty(B_r(x_0))}+\|g\|_{L^\infty(B_r(x_0))}\big)+N\cF(r),
\end{aligned}
\end{equation}
where $x_0\in \cD'$ and $r\in (0, R_1]$ with $B_{r}(x_0)\subset \cD'$.
In the above,  $N$ depends only on $d$, $\nu$, $M$, $\varrho_0$, $\omega_{A^{\alpha\beta}}$, and $\gamma$.

{\bf{Step 2}}.
Let $x_0\in \cD'$ and  $r\in (0, R_1]$ with $B_r(x_0)\subset \cD'$, and fix a coordinate system associated with $x_0$ satisfying  $({\bf{A}}1)$ and $({\bf{A}}2)$.
We claim that
\begin{equation}		\label{210429@A1}
\begin{aligned}
&|(D_{x'}u(x_0), U(x_0))-(D_{x'}u(y_0), U(y_0))|\\
&\lesssim r^{-d}\big(\|Du\|_{L^1(B_r(x_0))}+\|p\|_{L^1(B_r(x_0))}\big)\cE(|x_0-y_0|)\\
&\quad +\big(\|f_\alpha\|_{L^\infty(B_r(x_0))}+\|g\|_{L^\infty(B_r(x_0))}\big)\cE(|x_0-y_0|)\\
&\quad +\cF(r)\cE(|x_0-y_0|)+\cF(|x_0-y_0|)
\end{aligned}
\end{equation}
 for any $y_0\in B_{r/4}(x_0)$, where
\begin{equation*}
\cE(|x_0-y_0|):=\bigg(\frac{|x_0-y_0|}{r}\bigg)^\gamma+\int_0^{|x_0-y_0|}\frac{\tilde{\omega}_{A^{\alpha\beta}}(t)+\tilde{\varrho}_1(t)}{t}\ dt.
\end{equation*}
Let $y_0\in B_{r/4}(x_0)$ and $\rho:=|x_0-y_0|$.
We consider the following two cases:
$$
B_\rho(y_0)\subset \widehat{\cD}_{i_0},  \quad B_\rho(y_0)\not\subset \widehat{\cD}_{i_0}.
$$

{\em{Case 1}}.
$B_{\rho}(y_0)\subset \widehat{\cD}_{i_0}$.
By the triangle inequality,  we have
$$
\begin{aligned}
&|(D_{x'}u(x_0), U(x_0))-(D_{x'}u(y_0), U(y_0))|^{\frac{1}{2}}\\
&\le |(D_{x'}u(x_0), U(x_0))-\Theta_{x_0}(x_0, \rho)|^{\frac{1}{2}}+|(D_{x'}u(x), U(x))-\Theta_{x_0}(x_0,\rho)|^{\frac{1}{2}}\\
&\quad +|(D_{x'}u(y_0), U(y_0))-\Theta_{x_0}(y_0,\rho)|^{\frac{1}{2}}+|(D_{x'}u(x), U(x))-\Theta_{x_0}(y_0,\rho)|^{\frac{1}{2}}.
\end{aligned}
$$
for all  $x\in B_\rho(x_0) \cap B_\rho(y_0)$.
Taking the average over $x\in B_\rho(x_0)\cap B_\rho(y_0)$ and then taking the square, we obtain that
$$
|(D_{x'}u(x_0), U(x_0))-(D_{x'}u(y_0), U(y_0))|\lesssim I_1+I_2,
$$
where
$$
\begin{aligned}
I_1&=|(D_{x'}u(x_0), U(x_0))-\Theta_{x_0}(x_0, \rho)|+\Phi_{x_0}(x_0, \rho),\\
I_2&=|(D_{x'}u(y_0), U(y_0))-\Theta_{x_0}(y_0, \rho)|+\Phi_{x_0}(y_0, \rho).
\end{aligned}
$$
Note that by \eqref{210226@A1},  we have
$$
I_1\lesssim \sum_{i=0}^\infty \Phi_{x_0}(x_0, \kappa^i \rho).
$$
It follows from Lemma \ref{210127@lem1} $(ii)$ that
$$
\lim_{i\to \infty}\Phi_{x_0}(y_0, \kappa^{i}r)=0.
$$
Then by replicating a similar argument that used in \eqref{210226@A1}, we obtain
$$
I_2\lesssim \sum_{i=0}^\infty\Phi_{x_0}(y_0, \kappa^{i}\rho).
$$
Therefore, by Lemma \ref{210127@lem1}, \eqref{210226@A2}, and \eqref{210202@eq1}, we get
\eqref{210429@A1}.

{\em{Case 2}}.
$B_\rho(y_0)\not\subset \widehat{\cD}_{i_0}$.
In this case, for simplicity of notation, we assume that $y_0=0$.
Suppose that $0\in \widehat{\cD}_{i_1}\cup \partial \widehat{\cD}_{i_1}$ for some $i_1\in \{1,\ldots, \ell+1\}$ and denote by $\tilde{y}_0$ the closest point on $\partial \widehat{\cD}_{i_1}$ to the origin.
We also denote $\tilde{x}_0=(x_0', \chi_{i_0}(x_0'))$, which is the closest point on $\partial \widehat{\cD}_{i_0}$ to $x_0$.
Since
$|\tilde{y}_0|<\rho$ and $|\tilde{x}_0-x_0|<2\rho$, we have
\begin{equation}		\label{210323@eq1}
|\tilde{x}_0-\tilde{y}_0|\le |\tilde{x}_0-x_0|+|x_0|+|\tilde{y}_0|<4\rho<r\le R_1.
\end{equation}
Let
$$
y=\Lambda x, \quad x=\Lambda^{-1}y=\Gamma y,
$$
where $\Lambda$ is a $d\times d$ rotation matrix from the coordinate systems associated with $x_0$ to a coordinate system associated with the origin.
Then by \eqref{210323@eq1} and the same argument as in \cite[pp. 2465--2466]{MR3961984}, we see that
$$
|{\bf{I}}-\Gamma|\lesssim \varrho_1(4\rho),
$$
where ${\bf{I}}$ is the $d\times d$ identity matrix.
From the definition of $\tilde{\varrho}_1$ and \eqref{210226@A2}, it follows that
\begin{equation}		\label{210323@eq1a}
|{\bf{I}}-\Gamma|\lesssim \tilde{\varrho}_1(\rho)\lesssim \int_0^\rho \frac{\tilde{\varrho}_1(t)}{t}\,dt.
\end{equation}

Now we  set
$$
v(y)=\Lambda u(x), \quad \pi(y) =p(x),
$$
which satisfies
$$
\begin{cases}
D_\alpha (\cA^{\alpha\beta}D_\beta v)+\nabla \pi=D_\alpha F_\alpha,\\
\Div v=G,
\end{cases}
$$
where
$$
\cA^{\alpha\beta}(y)=\Lambda (\Lambda^{\alpha k}\Lambda^{\beta l}A^{kl}(x))\Gamma,
$$
$$
(F_1, \ldots, F_d)(y)=\Lambda (f_1,\ldots, f_d)(x)\Gamma, \quad \quad G(y)=g(x).
$$
We also denote
$$
V=\cA^{d\beta}D_\beta v+\pi e_d-F_d.
$$
By the triangle inequality, we have
$$
\begin{aligned}
&|(D_{x'}u(x_0), U(x_0))-(D_{x'}u(0), U(0))|^{\frac{1}{2}}\\
&\le |(D_{x'}u(x_0), U(x_0))-\Theta_{x_0}(x_0, \rho)|^\frac{1}{2}+|(D_{x'}u(x), U(x))-\Theta_{x_0}(x_0, \rho)|^\frac{1}{2}\\
&\quad +|\Gamma(D_{y'}v(0), V(0))-\Gamma\Theta_{0}(0, \rho)|^{\frac{1}{2}}+|\Gamma (D_{y'}v(\Lambda x), V(\Lambda x))-\Gamma \Theta_{0}(0,\rho)|^{\frac{1}{2}}\\
&\quad +|(D_{x'}u(0), U(0))-\Gamma (D_{y'}v(0), V(0))|^{\frac{1}{2}}\\
&\quad +|(D_{x'}u(x), U(x))-\Gamma(D_{y'}v(\Lambda x), V(\Lambda x))|^{\frac{1}{2}}
\end{aligned}
$$
for any  $x\in B_\rho(x_0) \cap B_\rho(0)$, where $\Gamma(D_{y'}v,V):=(\Gamma D_{y'}v,\Gamma V)$.
Taking the average over $x\in B_\rho(x_0)\cap B_\rho(0)$ and then taking the square, we obtain that
\begin{equation}		\label{210323@eq5}
|(D_{x'}u(x_0), U(x_0))-(D_{x'}u(0), U(0))|\lesssim J_1+J_2+J_3,
\end{equation}
where
$$
\begin{aligned}
J_1&=|(D_{x'}u(x_0), U(x_0))-\Theta_{x_0}(x_0, \rho)|+\Phi_{x_0}(x_0, \rho),\\
J_2&=|(D_{y'}v(0), V(0))-\Theta_{0}(0, \rho)|+\Phi_{0}(0, \rho),\\
J_3&=\operatorname*{ess\,sup}_{x\in B_\rho(x_0)\cap B_\rho(0)} |(D_{x'}u(x), U(x))-\Gamma(D_{y'}v(\Lambda x), V(\Lambda x))|.
\end{aligned}
$$
Note that $J_1$ and $J_2$ can be estimated by Lemma \ref{210127@lem1} $(i)$, \eqref{210226@A2}, and \eqref{210202@eq1} in the same way as in  {\em{Case 1}}.
For the estimate of $J_3$, we observe that
$$
D_{x'}u(x)-\Gamma D_{y'} v(\Lambda x)=D_{x}u(x)-\Gamma D_y v(\Lambda x){\bf I}_0=D_{x}u(x) ({\bf I}-\Gamma){\bf I}_0,
$$
where ${\bf I}_0=(I_0^{\alpha\beta})$ is a $d\times (d-1)$ matrix with
$$I_0^{\alpha\beta}=\delta_{\alpha\beta}~\mbox{for}~\alpha,\beta=1,\dots,d-1;\quad I_0^{d\beta}=0~\mbox{for}~\beta=1,\dots,d-1,$$
and
$$
\begin{aligned}
U(x)-\Gamma V(\Lambda x)&=(1-\Lambda^{k\alpha})A^{\alpha\beta}(x) D_\beta u(x)\\
&\quad +p(x)(I-\Gamma)e_d+(f_1,\ldots, f_d)(x)({\bf{I}}-\Gamma)^{\cdot d},
\end{aligned}
$$
where $({\bf{I}}-\Gamma)^{\cdot d}$ is the $d$th column of ${\bf{I}}-\Gamma$.
Hence by \eqref{210226@A2} and \eqref{210323@eq1a}, we have
$$
\begin{aligned}
J_3&\lesssim \big(\|Du\|_{L^\infty(B_{r/4}(x_0))}+\|p\|_{L^\infty(B_{r/4}(x_0))}+\|f_\alpha\|_{L^\infty(B_{r}(x_0))}\big)\int_0^\rho \frac{\tilde{\varrho_1}(t)}{t}\,dt\\
&\lesssim r^{-d}\big(\|Du\|_{L^1(B_r(x_0))}+\|p\|_{L^1(B_r(x_0))}\big)\int_0^\rho \frac{\tilde{\varrho_1}(t)}{t}\,dt\\
&\quad +\big(\|f_\alpha\|_{L^\infty(B_r(x_0))}+\|g\|_{L^\infty(B_r(x_0))}\big)\int_0^\rho \frac{\tilde{\varrho_1}(t)}{t}\,dt+\cF(r)\int_0^\rho \frac{\tilde{\varrho_1}(t)}{t}\,dt.
\end{aligned}
$$
Using this together with the estimates $J_1$ and $J_2$, we get \eqref{210429@A1} from \eqref{210323@eq5}.

Note that the piecewise continuity of $(Du, p)$ follows from the estimate \eqref{210429@A1} combined with the fact that the coefficients and data are piecewise continuous.
Indeed, by using the relations \eqref{210323@EQ2}, \eqref{210420@A1},  and \eqref{210323@EQ3}, and using the triangle inequality, we have that
$$
|D_d u^d(x_0)-D_{d}u^d(y_0)|\le |D_{x'}u(x_0)-D_{x'}u(y_0)|+|g(x_0)-g(y_0)|,
$$
$$
\begin{aligned}
|\cX(x_0)-\cX(y_0)|&\lesssim_{d, \nu}|(D_{x'}u(x_0), U(x_0))-(D_{x'}u(y_0), U(y_0))|\\
&\quad +\big(\|Du\|_{L^\infty(B_{r/4}(x_0))}+\|p\|_{L^\infty(B_{r/4}(x_0))}\big)|A^{\alpha\beta}(x_0)-A^{\alpha\beta}(y_0)|\\
&\quad +\|f_\alpha\|_{L^\infty(B_{r/4}(x_0))}|A^{\alpha\beta}(x_0)-A^{\alpha\beta}(y_0)|\\
&\quad +|D_d u^d(x_0)-D_du^d(y_0)|+|f_\alpha(x_0)-f_\alpha(y_0)|,
\end{aligned}
$$
where $\cX=(D_du^1,\ldots, D_d u^{d-1})^{\top}$,
and
$$
\begin{aligned}
|p(x_0)-p(y_0)|
&\lesssim_{d, \nu}|(D_{x'}u(x_0), U(x_0))-(D_{x'}u(y_0), U(y_0))|\\
&\quad + \|Du\|_{L^\infty(B_{r/4}(x_0))}|A^{\alpha\beta}(x_0)-A^{\alpha\beta}(y_0)|+|f_\alpha(x_0)-f_\alpha(y_0)|.
\end{aligned}
$$
Therefore, by \eqref{210202@eq1} and \eqref{210429@A1}, we obtain that
\begin{align}\label{210429@B1}
&|(Du(x_0), p(x_0))-(Du(y_0), p(y_0))|\nonumber\\
&\le N  r^{-d}\big(\|Du\|_{L^1(B_r(x_0))}+\|p\|_{L^1(B_r(x_0))}\big)
\big(\cE(|x_0-y_0|)+|A^{\alpha\beta}(x_0)-A^{\alpha\beta}(y_0)|\big)\nonumber\\
&\quad + N\big(\|f_\alpha\|_{L^\infty(B_r(x_0))}+\|g\|_{L^\infty(B_r(x_0))}\big)
\big(\cE(|x_0-y_0|)+|A^{\alpha\beta}(x_0)-A^{\alpha\beta}(y_0)|\big)\nonumber\\
&\quad + N\cF(r)\big(\cE(|x_0-y_0|)+|A^{\alpha\beta}(x_0)-A^{\alpha\beta}(y_0)|\big)+N\cF(|x_0-y_0|)\nonumber\\
&\quad + N|f_\alpha(x_0)-f_{\alpha}(y_0)|+N|g(x_0)-g(y_0)|
\end{align}
for any $x_0, y_0\in \cD'$ and $r\in (0, R_1]$ satisfying $y_0\in B_{r/4}(x_0)\subset B_r(x_0)\subset \cD'$,
which gives the piecewise continuity of $(Du, p)$.

{\bf{Step 3}}.
In this step, we derive the corresponding estimate of \eqref{210429@B1} under the additional stronger \eqref{210215@eq3}. We again let $x_0\in \cD'$ and  $r\in (0, R_1]$ with $B_r(x_0)\subset \cD'$, and fix a coordinate system associated with $x_0$ satisfying $({\bf{A}}1)$ and $({\bf{A}}2)$.
To present the precise dependence of the constant in the estimates, we assume that
\begin{equation}		\label{210215@A1}
\varrho_0(r)\le K_0 r^{\frac{\gamma_0}{1-\gamma_0}}, \quad \omega_{A^{\alpha\beta}}(r)\le K_0 r^{\gamma_0}, \quad \omega_{f_\alpha}(r)+\omega_g(r)\le K_1r^{\gamma_0}
\end{equation}
for some constants $K_0,K_1>0$.
Thus  if $f_\alpha$ and $g$ are in $C^{\gamma_0}(\overline{\cD}_i)$ for each $i\in \{1,\ldots, M\}$, then
$K_1$ can be regarded as
$$
\max_{1\le i\le M}\big\{[f_\alpha]_{C^{\gamma_0}(\overline{\cD}_i)}+[g]_{C^{\gamma_0}(\overline{\cD}_i)}\big\}.
$$
From \cite[Lemma 5.1]{MR1770682} it follows that for any $r\in (0, R_1]$,
$$
r^{-d}|(\widehat{\cD}_i \setminus \Omega_i)\cap B_{r}(x_{0})|\lesssim_{d, M, K_0, \gamma_0} r^{\gamma_0}=:\varrho_1(r).
$$
Hence we have
$$
\tilde{\omega}_{A^{\alpha\beta}}(r)+\tilde{\varrho}_1(r)\lesssim_{d, M, K_0, \gamma_0} r^{\gamma_0}
$$
and
$$
\tilde{\omega}_{f_\alpha}(r)+\tilde{\omega}_g(r)\lesssim_{d, M, K_0, \gamma_0} K_1r^{\gamma_0}.
$$
Therefore by \eqref{210429@B1} with $\gamma=\frac{1+\gamma_0}{2}$, we conclude  that
\begin{align}\label{210209@C1}
&|(Du(x_0), p(x_0))-(Du(y_0), p(y_0))|\nonumber\\
&\le N r^{-d}\big(\|Du\|_{L^1(B_r(x_0))}+\|p\|_{L^1(B_r(x_0))}\big)\bigg(\frac{|x_0-y_0|^{\gamma_0}}{r^{\gamma_0}}+|A^{\alpha\beta}(x_0)-A^{\alpha\beta}(y_0)|\bigg)\nonumber\\
&\quad +N \big(\|f_\alpha\|_{L^\infty(B_r(x_0))}+\|g\|_{L^\infty(B_r(x_0))}\big)\bigg(\frac{|x_0-y_0|^{\gamma_0}}{r^{\gamma_0}}+|A^{\alpha\beta}(x_0)-A^{\alpha\beta}(y_0)|\bigg)\nonumber\\
&\quad +N K_1\big(|x_0-y_0|^{\gamma_0}+|A^{\alpha\beta}(x_0)-A^{\alpha\beta}(y_0)|\big)\nonumber\\
&\quad +N |f_\alpha(x_0)-f_{\alpha}(y_0)|+N|g(x_0)-g(y_0)|,
\end{align}
where
 $N=N(d, \nu, M, K_0,\gamma_0)$.
 We can see from \eqref{210209@C1} that if $x_0$ and $y_0$ are in the same subdomain, then the estimate of the modulus of continuity of $(Du, p)$ is established.

{\bf{Step 4}}.
In this last step, we prove the local boundedness of $(Du, p)$.
We first observe that
\begin{equation}		\label{210411@A1}
(Du, p)\in L^q_{\rm{loc}}(\cD)^d\times L^q_{\rm{loc}}(\cD) \quad \text{for any }\, q<\infty.
\end{equation}
Indeed, since $(u, p)$ satisfies \eqref{210201@eq4}, where the coefficients $A^{\alpha\beta}$ are of variably partially small bounded mean oscillation (variably partially BMO) satisfying \cite[Assumption 2.2 $(\rho)$ $(i)$]{MR3758532} for any $\rho>0$ and the data $f_\alpha$, $g$ are bounded, by applying a local version of \cite[Theorem 2.4]{MR3758532} combined with a bootstrap argument, we get \eqref{210411@A1}.

Due to the regularity result in \cite{MR3912724}, where the authors proved $W^{1,\infty}$-estimates for solutions to Stokes systems with (partially) Dini mean oscillation coefficients in a ball, it suffices to show that for $x_0=(x_0', x_0^d)\in \partial \cD_i$, $i\in \{1,\ldots, M-1\}$, there is a neighborhood of $x_0$ in which $(Du, p)$ is bounded.
Recall that $x_0$ belongs to the boundaries of at most two of the subdomains.
Thus we can find a small $r_0>0$ and a $C^{1,\rm{Dini}}$ function, say $\chi:\bR^{d-1}\to \bR$, such that $B_{r_0}(x_0)$ is divided into two disjoint subdomains separated by $\chi$ and
$|\nabla_{x'}\chi(x_0')|=0$ in a coordinate system.
Here, we choose $r_0$ small enough so that
\begin{equation}		\label{210411@A2}
|\nabla_{x'} \chi(x')|\le \mu_0\quad \text{if }\, |x'-x_0'|\le r_0,
\end{equation}
where $\mu_0>0$ is a constant to be chosen below.
Without loss of generality, we assume that $x_0=(0',0)$ and $\chi(0')=0$.
For sufficiently small $\varepsilon>0$, we let $\chi_{\varepsilon}$ be a standard mollification of $\chi$ with respect to $x'$.
We also let $\phi\in C^\infty_0(B_1)$ be a smooth non-negative function with unit integral, and  define piecewise mollifications of  $A^{\alpha\beta}$ by
$$
A^{\alpha\beta}_{\varepsilon}(x)=
\int_{B_\varepsilon(x_\varepsilon)}\phi_{\varepsilon}(x_\varepsilon-y) A^{\alpha\beta}(y)\,dy=\int_{B_\varepsilon} \phi_{\varepsilon}(y)A^{\alpha\beta}(x_{\varepsilon}-y)\,dy,
$$
where $\phi_{\varepsilon}(x)=\varepsilon^{-d}\phi (x/\varepsilon)$ and
$$
x_{\varepsilon}=
\begin{cases}
x+\lambda \varepsilon e_d &\quad \text{if }\, x^d>\chi_{\varepsilon}(x'),\\
x-\lambda \varepsilon e_d &\quad \text{if }\, x^d< \chi_{\varepsilon}(x').
\end{cases}
$$
Here $\lambda$ is large enough, say $\lambda=\mu_0+1$.
Similarly, we define $f_{\alpha, \varepsilon}$ and $g_{\varepsilon}$.
Then the piecewise mollifications are piecewise Dini mean oscillation in $B_{r_0}$ with
$$
\omega_{\bullet_{\varepsilon}}(r)\le \omega_{\bullet}(r).
$$
Let $(\tilde{u}_\varepsilon, \tilde{p}_\varepsilon)$ be the weak solution in $W^{1,2}_{0}(B_{r_0})^d\times \tilde L^2(B_{r_0})$  to the problem
\begin{equation}		\label{210420@eq1}
\begin{cases}
D_\alpha(A^{\alpha\beta}_{\varepsilon}D_\beta \tilde{u}_\varepsilon)+\nabla \tilde{p}_\varepsilon=D_\alpha (f_\alpha-f_{\alpha, \varepsilon})+D_\alpha(({A}^{\alpha\beta}_{\varepsilon}-A^{\alpha\beta}) D_\beta u),\\
\Div \tilde{u}_\varepsilon=g-g_\varepsilon-(g-g_\varepsilon)_{B_{r_0}}.
\end{cases}
\end{equation}
Since $f_{\alpha, \varepsilon}\to f_{\alpha}$ in $L^2$, $g_\varepsilon\to g$ in $L^2$, and ${A}^{\alpha\beta}_{\varepsilon}\to A^{\alpha\beta}$ a.e., by the dominated convergence theorem,
the right-hand sides of \eqref{210420@eq1} go to zero in $L^2$ as $\varepsilon\to 0^+$. By the $W^{1,2}$-estimate, we see that
$$
\|D\tilde{u}_{\varepsilon}\|_{L^2(B_{r_0})}+\|\tilde{p}_{\varepsilon}\|_{L^2(B_{r_0})}\to 0 \quad \text{as }\, \varepsilon\to 0^+,
$$
and thus, there is a subsequence, still denoted by $(\tilde{u}_{\varepsilon}, \tilde{p}_{\varepsilon})$, such that $|D\tilde{u}_{\varepsilon}|+|\tilde{p}_{\varepsilon}|\to 0$ a.e. in $B_{r_0}$.

Now we set $(u_\varepsilon, p_\varepsilon)=(u-\tilde{u}_{\varepsilon}, p-\tilde{p}_\varepsilon)\in W^{1,2}(B_{r_0})^d\times L^2(B_{r_0})$, which satisfies
\begin{equation}		\label{210421@B1}
\begin{cases}
D_\alpha(A^{\alpha\beta}_{\varepsilon}D_\beta {u}_\varepsilon)+\nabla {p}_\varepsilon=D_\alpha f_{\alpha, \varepsilon},\\
\Div u_\varepsilon=g_\varepsilon+(g-g_\varepsilon)_{B_{r_0}}
\end{cases}
\end{equation}
in $B_{r_0}$.
By the same reasoning as in \eqref{210411@A1}, it holds that
$$
(Du_{\varepsilon}, p_{\varepsilon})\in L^q_{\rm{loc}}(B_{r_0})^{d\times d}\times L^q_{\rm{loc}}(B_{r_0}) \quad \text{for any $q<\infty$}.
$$
We shall prove that $(Du_{\varepsilon}, p_{\varepsilon})$ is  bounded near the origin so that  \eqref{210202@eq1} can be applied to the above system, which gives uniform $L^\infty$-estimate of $(Du_{\varepsilon}, p_{\varepsilon})$.
To this end, we fix $\varepsilon>0$ and let
$$
y=\Lambda(x)=(x', x^d-\chi_\varepsilon(x')), \quad x=\Lambda^{-1}(y)=\Gamma(y)=(y', y^d+\chi_\varepsilon(y')).
$$
Then $(v(y), \pi (y))=(u_\varepsilon(x), p_\varepsilon(x))$ satisfies
\begin{equation}		\label{210410@eq1}
\begin{cases}
D_\alpha(\cA^{\alpha\beta} D_\beta v)+\nabla \pi=D_{\alpha}F_{\alpha}+D_d(\pi b),\\
\Div v=G+D_d v\cdot b
\end{cases}
\end{equation}
in $B_{r_1}$ with a sufficiently small $r_1>0$ so that $\overline{B_{r_1}}\subset \Lambda(B_{r_0})$,
where
$$
\cA^{\alpha\beta}(y)=D_l \Lambda^\beta D_k \Lambda^\alpha A_\varepsilon^{kl}(x), \quad F_\alpha(y)=D_k \Lambda^\alpha f_{k, \varepsilon}(x),
$$
$$
G(y)=g_\varepsilon(x)+(g-g_\varepsilon)_{B_{r_0/2}}, \quad b(y)=\big(D_1\chi_\varepsilon(y'), \ldots, D_{d-1}\chi_\varepsilon(y'), 0\big).
$$
Note that the coefficients and data are of partially Dini mean oscillation  in $B_{r_1}$ except $\pi b$ and $D_d v\cdot b$, which are only known to be in $L^q(B_{r_1})$ for $q<\infty$.
Thus we are not able to apply the result in \cite[Theorem 2.2]{MR3912724} to \eqref{210410@eq1} directly.
To overcome this  difficulty, we use the following fixed point argument.

Let $\eta$ be an infinitely differentiable function in $\bR^d$ such that
$$
0\le \eta\le 1, \quad \eta\equiv 1 \, \text{ in }\, B_{r_1/2}, \quad \operatorname{supp}\eta\subset B_{r_1}.
$$
Then we see that  $(\eta v, \eta \pi)$ satisfies
\begin{equation}		\label{210410@eq2}
\begin{cases}
D_\alpha({\cA}^{\alpha\beta}D_\beta(\eta v))+\nabla (\eta \pi)=D_\alpha\tilde{F}_\alpha+\tilde{F}+D_d (\eta \pi b),\\
\Div (\eta v)=\tilde{G}+D_d (\eta v)\cdot b
\end{cases}
\end{equation}
in $B_{r_1}$, where
$$
\tilde{F}_\alpha=\eta F_\alpha+\cA^{\alpha\beta} D_\beta \eta v, \quad
\tilde{F}=\cA^{\alpha\beta}D_{\alpha}\eta D_{\beta}v-D_\alpha \eta F_\alpha -D_d \eta \pi b+\nabla \eta \pi,
$$
$$
\tilde{G}=\eta G+\nabla \eta \cdot v-D_d \eta v\cdot b.
$$
For each positive integer $k$, let $(v^{(k)}, \pi^{(k)})$ be the weak solution in $W^{1,2}_0(B_{r_1})^d\times \tilde{L}^2(B_{r_1})$  to the problem
$$
\begin{cases}
D_\alpha(\cA^{\alpha\beta}D_\beta v^{(k)})+\nabla \pi^{(k)}=D_\alpha\tilde{F}_{\alpha}+\tilde{F}+D_d (\pi^{(k-1)} b),\\
\Div v^{(k)}=\tilde{G}+D_d v^{(k-1)}\cdot b-(\tilde{G}+D_dv^{(k-1)}\cdot b)_{B_{r_1}}
\end{cases}
$$
in $B_{r_1}$, where $(v^{(0)}, \pi^{(0)})=(0,0)$.
By applying the $W^{1,2}$-estimate to
\begin{equation}		\label{210421@A1}
\big(v^{(k+1)}-v^{(k)}, \pi^{(k+1)}-\pi^{(k)}\big)
\end{equation}
  and using \eqref{210411@A2},
we have
\begin{equation}		\label{210415@A2}
\begin{aligned}
&\|Dv^{(k+1)}-Dv^{(k)}\|_{L^2(B_{r_1})}+\|\pi^{(k+1)}-\pi^{(k)}\|_{L^2(B_{r_1})}\\
&\le N_0\big\|\big(D v^{(k)}-D v^{(k-1)}\big)  b\big\|_{L^2(B_{r_1})}+N_0\big\|\big(\pi^{(k)}-\pi^{(k-1)}\big) b\big\|_{L^2(B_{r_1})}\\
&\le\mu_0 N_0\big\|D v^{(k)}-D v^{(k-1)} \big\|_{L^2(B_{r_1})}+\mu_0 N_0\big\|\pi^{(k)}-\pi^{(k-1)} \big\|_{L^2(B_{r_1})},
\end{aligned}
\end{equation}
where the constant $N_0$ is independent of $\varepsilon$ and $\{(v^{(k)}, \pi^{(k)})\}$ .
We take $r_0$ sufficiently small so that \eqref{210411@A2} holds with $\mu_0=1/(2N_0)$. Then by the fixed point theorem, there exists
$$
(v^*, \pi^*)=(v^*_{\varepsilon}, \pi^*_{\varepsilon})\in W^{1,2}_0(B_{r_1})^d\times \tilde{L}^2(B_{r_1})
$$
such that as $k\rightarrow\infty$,
$$
v^{(k)}\to v^* \quad \text{in }\, W^{1,2}_0(B_{r_1}), \quad \pi^{(k)}\to \pi^* \quad \text{in }\, L^2(B_{r_1})
$$
and that in $B_{r_1}$,
\begin{equation}		\label{210415@A1}
\begin{cases}
D_\alpha(\cA^{\alpha\beta}D_\beta v^{*})+\nabla \pi^{*}=D_\alpha\tilde{F}_{\alpha}+\tilde{F}+D_d (\pi^* b),\\
\Div v^{*}=\tilde{G}+D_d v^{*}\cdot b-(\tilde{G}+D_dv^{*}\cdot b)_{B_{r_1}}.
\end{cases}
\end{equation}
From \eqref{210410@eq2} and \eqref{210415@A1}, it follows that in $B_{r_1}$,
$$
\begin{cases}
D_\alpha(\cA^{\alpha\beta}D_\beta ((\eta v-v^{*}))+\nabla (\eta\pi -(\eta \pi)_{B_{r_1}}-\pi^{*})=D_d ((\eta\pi-\pi^*) b),\\
\Div (\eta v-v^{*})=D_d (\eta v-v^{*})\cdot b+(\tilde{G}+D_dv^{*}\cdot b)_{B_{r_1}}.
\end{cases}
$$
Note that $D_d b=0$ and $(\tilde{G}+D_d (\eta v)\cdot b)_{B_{r_1}}=0$.
Hence by the $W^{1,2}$-estimate with the smallness of $b$, we obtain that
$$
\eta v=v^*, \quad \eta\pi-(\eta \pi)_{B_{r_2}}=\pi^*.
$$

Next, let $\rho_0\in (0, r_0]$ be small enough so that
\begin{equation}		\label{210422@C1}
|\nabla_{x'}\chi(x')|\le \mu_1 \quad \text{if }\, |x'|\le \rho_0,
\end{equation}
where $\mu_1$ is a constant to be chosen below.
We also let $\rho_1\in (0, r_1]$ such that $\overline{B_{\rho_1}}\subset \Lambda (B_{\rho_0})$.
Since $\cA^{\alpha\beta}$, $\tilde{F}_{\alpha}$, and $\tilde{G}$ are partially H\"older continuous with respect to $y'$, $\tilde{F}_d\in L^\infty(B_{\rho_1})$,  and $\tilde{F}\in L^q(B_{\rho_1})$ for all $q<\infty$, by applying \cite[Theorem 2.2 (b) and Remark 2.4]{MR3912724} combined with covering and scaling arguments, we obtain that
$$
(Dv^{(1)}, \pi^{(1)})\in L^{\infty}(B_{\rho})^{d\times d}\times L^\infty(B_\rho) \quad \text{for all $\rho<\rho_1$}.
$$
Moreover,
$$
\cA^{d\beta} D_\beta v^{(1)}+\pi^{(1)}e_d\in C^{\delta}_{x'}(B_\rho)^d, \quad  D_{x'}v^{(1)}\in C^\delta(B_\rho)^{d\times (d-1)},
$$
from which  we get
$$
(Dv^{(1)}, \pi^{(1)})\in C^{\delta}_{x'}(B_{\rho})^{d\times d}\times C^\delta_{x'}(B_\rho) \quad \text{for all }\delta\in (0,1).
$$
Repeating this procedure, we see that
$$
(Dv^{(k)}, \pi^{(k)})\in \big(L^{\infty}(B_{\rho})^{d\times d}\times L^\infty(B_\rho)\big)\cap \big(C^{\delta}_{x'}(B_{\rho})^{d\times d}\times C^\delta_{x'}(B_\rho)\big)
$$
for any positive integer $k$.
Hence, from the estimates in the proof of \cite[Theorem 2.2 (b)]{MR3912724} applied to \eqref{210421@A1} with covering and scaling arguments,
we deduce that for any $0<s<\rho<\rho_1$,
$$
\begin{aligned}
&\|Dv^{(k+1)}-Dv^{(k)}\|_{L^\infty(B_{s})}+\|\pi^{(k+1)}-\pi^{(k)}\|_{L^\infty(B_{s})}\\
&+(\rho-s)^{\delta}\Big(\big[Dv^{(k+1)}-Dv^{(k)}\big]_{C^\delta_{x'}(B_{s})}+\big[\pi^{(k+1)}-\pi^{(k)}\big]_{C^\delta_{x'}(B_{s})}\Big)\\
&\le N_1(\rho-s)^{-d}\Big(\|Dv^{(k+1)}-Dv^{(k)}\|_{L^1(B_\rho)}+\|\pi^{(k+1)}-\pi^{(k)}\|_{L^1(B_\rho)}\Big)\\
&\quad +N_1\Big(\big\|\big(D v^{(k)}-D v^{(k-1)}\big)  b\big\|_{L^\infty(B_{\rho})}+\big\|\big(\pi^{(k)}-\pi^{(k-1)}\big) b\big\|_{L^\infty(B_{\rho})}\Big)\\
&\quad +N_1(\rho-s)^\delta\Big(\big[\big(D v^{(k)}-D v^{(k-1)}\big)  b\big]_{C^\delta_{x'}(B_{\rho})}+\big[\big(\pi^{(k)}-\pi^{(k-1)}\big) b\big]_{C^\delta_{x'}(B_{\rho})}\Big)\\
&\le \mu_0 N_0N_1(\rho-s)^{-d/2}\Big(\|Dv^{(k)}-Dv^{(k-1)}\|_{L^2(B_{r_1})}+\|\pi^{(k)}-\pi^{(k-1)}\|_{L^{2}(B_{r_1})}\Big)\\
&\quad +\mu_1 N_1\Big(\|D v^{(k)}-D v^{(k-1)}\|_{L^\infty(B_{\rho})}+\|\pi^{(k)}-\pi^{(k-1)}\|_{L^\infty(B_{\rho})}\Big)\\
&\quad +\mu_1 N_1(\rho-s)^\delta\Big(\big[D v^{(k)}-D v^{(k-1)}\big]_{C^\delta_{x'}(B_{\rho})}+\big[\pi^{(k)}-\pi^{(k-1)} \big]_{C^\delta_{x'}(B_{\rho})}\Big),
\end{aligned}
$$
where we used \eqref{210411@A2}, \eqref{210415@A2}, and \eqref{210422@C1} in the second inequality.
Note that the constant $N_1$ is independent of $\{(v^{(k)}, \pi^{(k)})\}$, but it may depend on $\varepsilon$.
By choosing $\rho_0$ sufficiently small, which (and also $\rho_1$) may depend on $\varepsilon$, and following a standard iteration argument, we get uniform $L^\infty$ bounds of $Dv^{(k)}$ and $\pi^{(k)}$ in $B_{\rho_1/2}$.
Thus the functions
$$
Dv(y)=Dv^*(y), \quad \pi(y)=\pi^*(y),
$$
and hence $Du_{\varepsilon}(x)$ and $p_{\varepsilon}(x)$ are bounded in a neighborhood of the origin with a radius depending also on $\varepsilon$.
It is easy to check that the same argument as above still works at every point near the origin, for instance, in $B_{r_0/2}$, where $r_0$ is the constant from the beginning of this step, which is independent of $\varepsilon$.
Therefore,
$$
(Du_{\varepsilon}, p_{\varepsilon})\in L^\infty(B_{r_0/2})^{d\times d}\times L^\infty(B_{r_0/2}).
$$
Now we can apply the a priori estimate in {\bf Step 1} to \eqref{210421@B1} to get  uniform $L^\infty$-bounds of
$(Du_{\varepsilon}, p_{\varepsilon})$, and then, take the limit $\varepsilon\to 0^+$ to obtain the boundedness of the limit function $(Du, p)$ in $B_{r_0/2}$.
The theorem is proved.
\end{proof}

We conclude the proof of Theorem \ref{M1}  with the following remark.

\begin{remark}		\label{210219@rmk1}
As mentioned in Remark \ref{210210@rmk1}, the regularity results in Theorem \ref{M1} can be extended to  weak solutions of
$$
\begin{cases}
\mathcal{L}u+\nabla p=D_{\alpha}f_{\alpha}+f&\quad\mbox{in }~\cD,\\
\Div u=g&\quad\mbox{in }~\cD,
\end{cases}
$$
where $f\in L^{s}(\cD)^d$ with $s>d$.
In this case, the upper bounds of the $L^\infty$-norm of $(Du, p)$ and the modulus of continuity of $(D_{x'}u, U)$ can be derived as follows.

Let $x_0\in \cD'$ and $r\in (0, R_1]$ such that $B_r(x_0)\subset \cD'$.
Due to the solvability of  the
divergence equation (see, for instance, \cite[Lemma 3.1]{MR3809039}), there exist $h_\alpha\in W^{1,s}(B_r(x_0))^d$, $\alpha\in \{1,2,\ldots,d\}$, such that
$$
\sum_{\alpha=1}^dD_\alpha h_\alpha=f \quad \text{in }\, B_r(x_0)
$$
and
$$
(h_\alpha)_{B_r(x_0)}=0, \quad \|D h_\alpha\|_{L^s(B_r(x_0))}\lesssim_{d,s}\|f\|_{L^s(B_r(x_0))}.
$$
Then $(u, p)$ satisfies
$$
\begin{cases}
\mathcal{L}u+\nabla p=D_{\alpha}(f_{\alpha}+h_\alpha)&\quad\mbox{in }~B_r(x_0),\\
\Div u=g&\quad\mbox{in }~B_r(x_0),
\end{cases}
$$
where, by both Morrey and Poincar\'e inequalities,
$$
r^{1-d/s}[h_\alpha]_{C^{1-d/s}(B_r(x_0))}+\|h_\alpha\|_{L^\infty(B_r(x_0))}\lesssim r^{1-d/s}\|f\|_{L^s(B_r(x_0))}.
$$
Thus by the same argument as in the proof of Theorem \ref{M1} with a fixed $\gamma\in\big(1-\frac{d}{s},1\big)$, we have
$$
\begin{aligned}
&\|Du\|_{L^\infty(B_{r/2}(x_0))}+\|p\|_{L^\infty(B_{r/2}(x_0))}\\
&\le Nr^{-d} \big(\|Du\|_{L^1(B_r(x_0))}+\|p\|_{L^1(B_r(x_0))}\big)\\
&\quad + N\big(\|f_\alpha\|_{L^\infty(B_r(x_0))}+\|g\|_{L^\infty(B_r(x_0))}\big)+N\cF(r)+Nr^{1-d/s}\|f\|_{L^s(B_r(x_0))},
\end{aligned}
$$
where $N=N(d, \nu, M, \varrho_0, \omega_{A^{\alpha\beta}},s)$.
Moreover,
for $y_0\in B_{r/4}(x_0)$, we obtain that
$$
\begin{aligned}
&|(D_{x'}u(x_0), U(x_0))-(D_{x'}u(y_0), U(y_0))|\\
&\le N r^{-d}\big(\|Du\|_{L^1(B_r(x_0))}+\|p\|_{L^1(B_r(x_0))}\big)\cE(|x_0-y_0|)\\
&\quad +N\big(\|f_\alpha\|_{L^\infty(B_r(x_0))}+\|g\|_{L^\infty(B_r(x_0))}\big)\cE(|x_0-y_0|)\\
&\quad +N(\cF(r)+r^{1-d/s}\|f\|_{L^s(B_r(x_0))})\cE(|x_0-y_0|)+N\cF(|x_0-y_0|)\\
&\quad +N \|f\|_{L^s(B_r(x_0))}|x_0-y_0|^{1-d/s}.
\end{aligned}
$$
\end{remark}

\subsection{Proof of Theorem \ref{M2}}		\label{S_3_2}
Note that $(u, p)$ satisfies
$$
\begin{cases}
\mathcal{L}u+\nabla p=D_{\alpha}f_{\alpha}+f&\quad\mbox{in }~\cD,\\
\Div u=g&\quad\mbox{in }~\cD,
\end{cases}
$$
where $f=-u^\alpha D_\alpha u$.
We consider two cases.

{\em{Case 1}}. $q>d$. In this case, by the Morrey–Sobolev embedding theorem, we see that $f\in L^q_{\rm{loc}}(\cD)^d$.
Thus the theorem follows from Remark \ref{210210@rmk1} applied to a slightly shrunk domain.

{\em{Case 2}}. $q\le d$.
From the first case,  it suffices to improve the regularity of $Du$ from $L^q$ to $L^{s}_{\rm{loc}}$ for some $s>d$.
Let $x_0\in \cD$.
We may assume that $x_0=0$ and $B_1\subset \cD$ after translating and scaling the coordinates.

We first derive an a priori estimate for $(Du, p)$ under the assumption that $(u, p)\in W^{1,q^*}(B_1)^d\times L^{q^*}(B_1)$, where $q^*$ is the Sobolev conjugate of $q$, i.e., $q^*=dq/(d-q)$ when $q<d$ and $q^*\in (q,\infty)$ is arbitrary when $q=d$.
Let $\eta$ be an infinitely differentiable function in $\bR^d$ such that
$$
0\le \eta\le 1, \quad \eta\equiv 1 \, \text{ in }\, B_{1/2}, \quad \operatorname{supp}\eta\subset B_1, \quad |\nabla \eta|\lesssim_d 1.
$$
We define an elliptic operator  $\tilde{\cL}$ by
$$
\tilde{\cL} u=D_\alpha (\tilde{A}^{\alpha\beta}D_\beta u),
$$
where $\tilde{A}^{\alpha\beta}=\eta A^{\alpha\beta}+\nu (1-\eta)\delta_{\alpha\beta}\bf{I}$.
Here, $\nu$ is the constant from \eqref{210210@eq4}, $\delta_{\alpha\beta}$ is the Kronecker delta symbol, and $\bf{I}$ is the $d\times d$ identity matrix.
Note that  $\tilde{A}^{\alpha\beta}$ and $\Omega=B_1$ satisfy \cite[Assumption 2.2 $(\rho)$]{MR3758532} for any $\rho>0$.
Therefore, the $W^{1,q^*}$-estimate in \cite[Theorem 2.4]{MR3758532} is available for $\tilde{\cL}$ on $\Omega=B_1$.

Now, for $r, R$ with $0<r<R\le 1/2$, let $\zeta=\zeta_{r, R}$ be an infinitely differentiable function in $\bR^d$ such that
$$
0\le \zeta \le 1, \quad \zeta \equiv 1 \, \text{ in }\, B_{r}, \quad \operatorname{supp}\zeta \subset B_R, \quad |\nabla \zeta|\lesssim_d (R-r)^{-1}.
$$
Then $(v, \pi)=(\zeta u, \zeta p)\in W^{1,q}_0(B_1)^d\times L^{q}(B_1)$ satisfies
\begin{align}\label{210210@eq5}
\begin{cases}
\tilde{\cL}v+\nabla \pi=F+D_{\alpha}F_{\alpha}&\quad\mbox{in }~B_1,\\
\Div v=G&\quad\mbox{in }~B_1,
\end{cases}
\end{align}
where
$$
F=D_\alpha \zeta A^{\alpha\beta} D_\beta u+\nabla \zeta p-D_\alpha \zeta f_\alpha-\zeta u^{\alpha}D_\alpha u,
$$
$$
F_\alpha=A^{\alpha\beta}  uD_\beta \zeta+\zeta f_\alpha, \quad G=\nabla \zeta \cdot u+\zeta g.
$$
Observe that $F_\alpha\in L^{q^*}(B_1)^d$, $G\in L^{q^*}(B_1)$, and
$$
\begin{aligned}
\|F\|_{L^{q}(B_1)}&\lesssim_{d, \nu} (R-r)^{-1}\big(\|Du\|_{L^{q}(B_R)}+\|p\|_{L^{q}(B_R)}\big)\\
&\quad +R(R-r)^{-1}\|f_\alpha\|_{L^{q^*}(B_R)}+\|u\|_{L^d(B_R)}\|Du\|_{L^{q^*}(B_R)}.
\end{aligned}
$$
Then by the $W^{1,q^*}$-solvability in  \cite[Theorem 2.4]{MR3758532}, \eqref{210210@eq5} also have a unique solution $(\tilde v, \tilde \pi)\in W^{1,q^*}_0(B_1)^d\times \tilde L^{q^*}(B_1)$, which is also in $W^{1,q}_0(B_1)^d\times \tilde L^{q}(B_1)$. By the uniqueness of $W^{1,q}_0(B_1)^d\times L^{q}(B_1)$ solutions, we get $(\tilde v,\tilde \pi)=(v,\pi-(\pi)_{B_1})$.
By applying the $W^{1,q^*}$-estimate in  \cite[Theorem 2.4]{MR3758532} to \eqref{210210@eq5} and using the above inequality, we obtain that
$$
\begin{aligned}
&\|Dv\|_{L^{q^*}(B_1)}+\|\pi-(\pi)_{B_1}\|_{L^{q^*}(B_1)}\\
&\le N\big(\|F\|_{L^{q}(B_1)}+\|F_\alpha\|_{L^{q^*}(B_1)}+\|G\|_{L^{q^*}(B_1)}\big)\\
&\le N_0(R-r)^{-1}\big(\|Du\|_{L^{q}(B_R)}+\|p\|_{L^{q}(B_R)}\big)+N_0(R-r)^{-1}\|u\|_{L^{q^*}(B_R)}\\
&\quad +N_0 R(R-r)^{-1}\|f_\alpha\|_{L^{q^*}(B_R)}+N_0\|g\|_{L^{q^*}(B_R)}+N_0\|u\|_{L^d(B_R)}\|Du\|_{L^{q^*}(B_R)},
\end{aligned}
$$
where $N_0=N_0(d, \nu, M, R_0, \varrho_0, \omega_{A^{\alpha\beta}}, q)$.
From the triangle and H\"older's inequalities, it follows  that
$$
\begin{aligned}
&\|Du\|_{L^{q^*}(B_r)}+\|p\|_{L^{q^*}(B_r)}\\
&\le \|Dv\|_{L^{q^*}(B_1)}+\|\pi-(\pi)_{B_1}\|_{L^{q^*}(B_1)}+N_1\|\pi\|_{L^1(B_1)}\\
&\le \|Dv\|_{L^{q^*}(B_1)}+\|\pi-(\pi)_{B_1}\|_{L^{q^*}(B_1)}+N_1\|p\|_{L^q(B_R)}.
\end{aligned}
$$
Then by taking $R_2\in (0,1/2]$ so that
$$
N_0\|u\|_{L^d(B_{R_2})}\le \varepsilon:=\frac{1}{8},
$$
we have
$$
\begin{aligned}
&\|Du\|_{L^{q^*}(B_r)}+\|p\|_{L^{q^*}(B_r)}\\
&\le (N_0+N_1)(R-r)^{-1}\big(\|Du\|_{L^{q}(B_R)}+\|p\|_{L^{q}(B_R)}\big)+N_0(R-r)^{-1}\|u\|_{L^{q^*}(B_R)}\\
&\quad +N_0 R(R-r)^{-1}\|f_\alpha\|_{L^{q^*}(B_R)}+N_0\|g\|_{L^{q^*}(B_R)}+\varepsilon \|Du\|_{L^{q^*}(B_R)}.
\end{aligned}
$$
Note that the above inequality holds for all $r, R$ with $0<r<R\le R_2$.
Therefore, by the well-known iteration argument, we conclude the following a priori estimate for $(Du, p)$:
\begin{equation}		\label{210212@A1}
\begin{aligned}
&\|Du\|_{L^{q^*}(B_{R/2})}+\|p\|_{L^{q^*}(B_{R/2})}\\
&\lesssim R^{-1}\big(\|Du\|_{L^{q}(B_R)}+\|p\|_{L^{q}(B_R)}\big)\\
&\quad +R^{-1}\|u\|_{L^{q^*}(B_R)}+\|f_\alpha\|_{L^{q^*}(B_R)}+\|g\|_{L^{q^*}(B_R)}
\end{aligned}
\end{equation}
for all $R\in (0, R_2]$.

We are ready to prove
\begin{equation}		\label{210210@D1}
D u\in L^s_{\rm{loc}}(\cD)^{d\times d} \quad \text{for some $s>d$}.
\end{equation}
From \eqref{210212@A1} and  a standard approximation argument, one can show that
$Du\in L^{q^*}_{\rm{loc}}(\cD)^{d\times d}$.
This yields \eqref{210210@D1} when $d/2<q\le d$ because $q^*>d$.
On the other hand, if  $q=d/2$, then since $Du\in L^{q_1}_{\rm{loc}}(\cD)^{d\times d}$ for all $q_1\le d$,
by applying the above regularity result again,  we get \eqref{210210@D1}.
We have thus proved the regularity results in the theorem.
The corresponding upper bounds of the $L^\infty$-norm of $(Du, p)$ and the modulus of continuity of $(D_{x'}u, U)$ can be derived as in Remark \ref{210219@rmk1}.
\qed

\bibliographystyle{plain}

\end{document}